\documentclass{amsart}
\usepackage{a4wide}
\usepackage{graphicx} 
\usepackage[english]{babel}
\usepackage{amsthm,amssymb}
\usepackage{amsmath}
\usepackage{amsfonts}
\usepackage{enumitem}
\usepackage{tikz}
\usepackage{pinlabel}
\usepackage{overpic}
\usepackage[hidelinks]{hyperref}

\usetikzlibrary{decorations.pathmorphing}
\usetikzlibrary{decorations.pathreplacing,intersections}
\tikzset{zigzag/.style={decorate, decoration=zigzag}}

\usepackage{xcolor}

\newtheorem{de}{Definition}[section]
\newtheorem{thm}[de]{Theorem}
\newtheorem{prop}[de]{Proposition}

\theoremstyle{remark}
\newtheorem{re}[de]{Remark}
\newtheorem{ob}[de]{Observation}

\newcommand{\Z}{\mathbb{Z}}

\title{Fibered ribbon pretzels}

\author{Ana G.\ Lecuona}

\address{%
School of Mathematics and Statistics, University of Glasgow, University Place, UK and
\newline
Aix Marseille Universit\'{e}, CNRS, Centrale Marseille, I2M, UMR 7373, 13453 Marseille, France}

\email{ana.lecuona@glasgow.ac.uk}

\author{Andy Wand}
\address{%
School of Mathematics and Statistics, University of Glasgow, University Place, UK }
\email{andy.wand@glasgow.ac.uk}

\begin{document}

\begin{abstract}
    We classify fibered ribbon pretzel knots up to mutation. The classification is complete, except perhaps for members of Lecuona's ``exceptional'' family of \cite{pretzel}.  The result is obtained by combining lattice embedding techniques with Gabai's classification of fibered pretzel knots, and exhibiting ribbon disks, some of which lie outside of known patterns for standard pretzel projections.  
\end{abstract}

\maketitle

\section{Introduction}

Given a knot, or family of knots, in the 3-sphere one can ask many interesting mathematical questions; this paper focuses on \emph{fiberedness} and \emph{sliceness} among pretzel knots. A knot is said to be fibered when its complement admits a nice fibration by surfaces over $S^1$, and slice if it bounds a smoothly embedded disc in the 4-ball whose boundary is the 3-sphere in which the knot lives.

Fibered knots have many well studied features: for example we have known since the 60s that the Alexander polynomial of a fibered knot is monic and has degree twice the genus of the fibered surface. While there are invariants, such as Heegaard Floer homology, which detect fiberedness, and `recipes' to construct all fibered knots, it remains in general difficult to determine fiberedness and the corresponding fibered surfaces for families of knots. Fibered knots are moreover at the intersection of various themes in low dimensional topology, such as contact topology, where they are used to define open book decompositions of manifolds which in turn carry the information of contact structures, or algebraic geometry, where we find fibered knots as links of plane curve singularities.

Slice knots have been at the forefront of knot theory since work of Fox and Milnor in the late 50s. There are many computable knot invariants which behave in a particular way for slice knots; to name a few: the signature vanishes, the determinant is a square, the Alexander polynomial has a particular factorization, Ozsv\'ath and Szab\'o's $\tau$ invariant vanishes, Rasmussen's $s$ invariant vanishes etc. While obstructions abound, by far the most common way to show that a knot is (smoothly) slice is to exhibit a slice disk. All the slice disks constructed in the literature have the property of being \emph{ribbon}; that is, they can all be projected into $S^3$ with exclusively \emph{ribbon singularities}. From the 4-dimensional point of view, the defining property of these disks is that they do not have local maxima for the radial function in $B^4$. The famous slice-ribbon conjecture proposes that every slice knot is in fact ribbon. 

The intersection of fiberedness and ribbon/sliceness is of particular interest, in large part due to work of Casson and Gordon \cite{casson1983loop}, who observe that existence of a ribbon disc implies a surjection at the level of $\pi_1$ from the knot complement to the disk complement, which of course by definition implies sliceness. Labelling this algebraic condition \emph{homotopically ribbon}, we have thus the following inclusions:
$$
\{\mathrm{ribbon}\}\subseteq\{\mathrm{homotopically\  ribbon}\}\subseteq\{\mathrm{slice}\}.
$$
Moreover, in the case of fibered knots, Casson and Gordon show that homotopical ribbonness admits a simple characterisation in terms of the fibration, namely that its monodromy extends over a handlebody. Investigation of the details of this result led Miller \cite[Remark~6.4]{Maggie} to study fibered ribbon 2-bridge knots, concluding that there are exactly 5 of these.

Beyond this first classification result, we are not aware of other results in the literature where there has been a similar attempt with other families of knots. This brings us to the third word in the title of this article: pretzel knots, which  are referenced back to at least Reidemeister in the 30s, and have long been a favorite family in which to test conjectures and find interesting examples of knotted phenomena. They provided the first examples of invertible and non reversible knots among other properties. They were classified by Kawauchi in the 80s and, relevant to this present work, Gabai in the 80s established which of them are fibered. The question of which pretzel knots are slice, in full generality, is still open, but, as further detailed in Section~\ref{s:background} quite a bit is known, providing a conjectural general picture. 

To state our results, some conventions, standard when discussing pretzel knots, need to be established. Firstly, as we describe in detail in the next section, such knots can be encoded by an ordered tuple of integers. Mutations of the knot correspond to re-orderings of these parameters; as mutation preserves most known obstructions to sliceness, we will only consider knots up to mutation. Secondly, a small family of pretzel knots, which we label \emph{exceptional}, while they can for the most part be dealt with individually, has so far resisted all systematic efforts at ribbon classification; accordingly we restrict ourselves to the non-exceptional knots. Finally, if a pretzel knot is not prime, it on the one hand is a connected sum of two-bridge knots which obviously fiber \cite[Remarks~6.2]{Gabai}, and thus itself fibers, while ribbonness of these was determined by Lisca \cite{lisca2007sums}.

With this in mind, we come to our main theorem: 
\begin{thm} \label{thm_main}
    Let $K$ be a prime non-exceptional pretzel knot. Then, $K$ is fibered and ribbon up to reordering of parameters if and only if one of the following holds: 
    \begin{enumerate}
        \item $K=\pm P(1,1,1,1,-3,-3,-3)$ (note that this is $\pm 10_{75}$ in the Rolfsen table), or
        \item $K=\pm P(q_1,\ldots,q_r,-q_1,\ldots,-q_r,k)$, with $q_i \geq 3$, $k$ even, or
        \item $K = \pm P(1,3,t+1,-4-t,q_1,\ldots,q_r,-q_1\ldots,-q_r$), with $q_i \geq3$, $r,t \geq0$, or 
        
        \item $K = \pm P(k,-k-1, q_1,\ldots,q_r,-q_1\ldots,-q_r$), with $q_i \geq3$, $r\geq 0$, and $1 < k < q_i$ for all $i$.
        
    \end{enumerate}
\end{thm}

We remark that the knot $\pm 10_{75}$ is actually mutation-invariant, so the ordering of parameters in the case (1) is arbitrary. Also, knots with parameters in cases (1), (3) and (4) are never exceptional. It is also worth noting that in the literature pretzel knots are often classified as \emph{odd}, i.e.\ defined by an odd number of odd parameters, or \emph{even} otherwise. In this language, Theorem \ref{thm_main} implies that $\pm 10_{75}$ is in fact the unique fibered ribbon odd pretzel knot, while there are infinitely many even ones. 

We want to stress that reordering of the parameters plays very differently with fiberedness and ribbonness, and we won't attempt to pin down the different possibilites. To showcase the intricacy of the problem let us highlight that
for example the pretzel knot $K=P(3,-3,5,-5,7,-7,4)$ is ribbon and fibered; its mutant $K'=P(3,5,7,-7,-5,-3,4)$ is ribbon, but not fibered; furthermore $K''=P(3,-5,7,-3,5,-7,4)$ is fibered but, as Nathan Dunfield pointed out in private communication, can be shown using SnapPy's implementation of Herald-Kirk-Livingston slice obstructions to not be ribbon \cite{SnapPy}, and finally the mutant $K'''=P(3,5,7,-3,-5,-7,4)$ is neither fibered nor ribbon. 

The strategy we follow to prove Theorem~\ref{thm_main} is as follows: taking Gabai's classification of fibered pretzel knots as a starting point, we determine which of them are potentially slice by obstructing their double covers from bounding rational homology balls using lattice embedding techniques derived from Donaldson's Theorem A. The potentially slice candidates are then shown to be ribbon by explicitly describing the disks. It is worth noting that it is customary in the literature to exclude or to treat independently the case in which parameters can include $\pm1$ or $0$. In fact, many previous results, particularly regarding sliceness and ribbonness properties, are either stated only for pretzel knots with no unitary parameters, or are treated separately \cite{pretzel,miller,bryant,long}. This article will also treat the unitary and non-unitary cases separately. If 0 is a parameter, it can only appear once, and the knot is a connected sum, so as mentioned above falls into known classifications.

The structure of the paper is as follows: in Section~\ref{s:background} we will establish definitions and conventions more precisely, and recall known results. In Section~\ref{s:lattice}, we will review the tools needed for our proofs, which will be left for Section~\ref{s:lemmas}. The proof of Theorem~\ref{thm_main} is obtained as a combination of results in this last section.

\noindent\textbf{Acknowledgements.} We would like to thank Nathan Dunfield for helpful correspondence, and the anonymous referee for a remarkably careful reading of our article and their comments and suggestions.

\section{Pretzel knots: fiberedness and sliceness}\label{s:background}

We start this section by collecting some well known facts about pretzel knots. A pretzel link is one which admits a projection as a series of $n$ pairs of strands with $|p_i|$ crossings in each pair, right handed if $p_i>0$ and left handed if $p_i<0$, joined by a series of arcs as in Figure~\ref{f:pretzel}.
\begin{figure}
\centering
\includegraphics[width=0.85\textwidth]{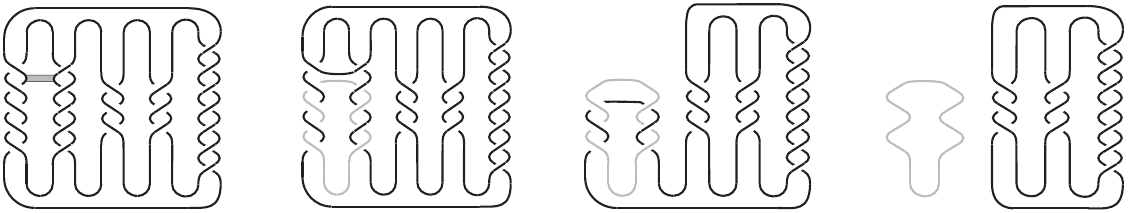}
\caption{The leftmost figure, without the gray band, is the standard projection of the pretzel knot $P(-5,5,-3,3,7)$. The gray band represents the ribbon move that can always be performed when we have adjacent parameters of opposite sign and same absolute value. The other figures depict the result of performing the ribbon move.}
\label{f:pretzel}
\end{figure}
It is thus determined by an ordered sequence of integer numbers ${p_1,...,p_n}$, which we refer to as its defining parameters. The result will be connected, and thus a knot, exactly when either $n$ and each $p_i$ is odd, or $n$ is arbitrary and exactly one $p_i$ is even. Following Gabai, we refer to the first case as Type 1, the second as Type 2 or Type 3 respectively depending on whether $n$ is odd or even.

Parameters of absolute value one, referred to as \emph{unitary}, can be permuted to any spot in a pretzel knot by flype moves. It follows that if $p_i=1$ and $p_j=-1$ in $K$, then $p_i$ and $p_j$ can be pairwise removed and $K=P(p_1,\dots,\widehat{p_i},\dots,\widehat{p_j},\dots p_n)$. For economy of notation, we will write $P([1^{\pm d}],p_{d+1},\dots,p_n)$ to indicate the pretzel knot with the first $d$ parameters equal to $\pm1$. We can, and from now on will, further assume that in $K$ we do not have indices $i$ and $j$ such that $p_i=\pm1$ and $p_j=\mp 2$: by flyping we can simplify such a pretzel knot to $K=P(p_1,\dots,\widehat{p_i},\dots,p'_j=\pm 2,\dots, p_n)$. The non-unitary parameters (i.e.\ with $|p_i|>1$) can be cyclically permuted and their order can be reversed without changing the knot type. Further permutations of the parameters might yield non isotopic pretzel knots. For example $P(3,5,7,2)$ is isotopic to $P(5,7,2,3)$ but not to $P(3,7,5,2)$. All changes of order in the parameters yield \emph{mutant} pretzel knots, which share the double covers of $S^3$ branched over them \cite{bedient}. The mirror of the pretzel knot $K=P(p_1,\dots,p_n)$ is $\bar K=-P(p_1,\dots,p_n)=P(-p_1,\dots,-p_n)$.

A complete list of fibred pretzel links, with a description of their fibered surfaces, was established by Gabai in \cite{Gabai}. In this article we only deal with knots, which simplifies Gabai's result substantially. We summarize here the results relevant to us. For convenience, unless otherwise specified, we will order the parameters of knots of Types 2 and 3 such that the only even parameter is the last one. For these knots, Gabai's classification uses an auxiliary pretzel link constructed as follows. For Type 2,
record the sign of the \emph{odd} non-unitary parameters as $a_i=\frac{p_i}{|p_i|}$ and to $K=P([1^{\pm d}],p_{d+1},\dots,p_{n-1},2m)$ associate the pretzel link $L'=P(-2a_{d+1},-2a_{d+2},\dots,-2a_{n-1},2m)$. For Type 3 the construction of $L'$ is very similar: this time we record the sign of \emph{all} non-unitary parameters and $L'$ becomes the link with only $\pm 2$ entries defined as $L'=P(-2a_{d+1},-2a_{d+2},\dots,-2a_{n-1},\frac{2m}{|m|})$.
Finally, we recall Gabai's theorem, adapted to our setting and notation:
\begin{thm}[Gabai, Theorem~6.7 in \cite{Gabai}]\label{t:Gabai}  
A pretzel knot $K$ is one of three Types.
    \begin{enumerate}[label=\emph{Type \arabic*:},wide, itemsep =0pt, topsep=0pt]
        \item $K$ fibers if and only if each $p_i$ is $\pm 1$ or $\mp 3$ and some $p_i=\pm1$. These knots are of the form $\pm P([1^d],-3,\dots,-3)$ with $d\geq 1$.

        \item  The pretzel knot $K$ falls into one of the following subcases:
             \begin{enumerate}[wide=\dimexpr\parindent+1em+\labelsep\relax, leftmargin=* ]
                \item[\emph{Type 2A:}] The numbers of positive and negative \emph{odd} parameters differ by two. Then $K$ fibers if and only if the absolute value of the unique even parameter is two.
                 \item[\emph{Type 2B:}] The numbers of positive and negative \emph{odd} parameters are equal and $L'\neq\pm P(2,-2,\dots,2,-2)$. Then $K$ fibers if and only if $L'=\pm P(2,-2,\dots,2,-2,n)$, $n\in\Z$ or $L'=\pm P(2,-2,\dots,2,-2,2,-4)$. 
                 \item[\emph{Type 2C:}] The numbers of positive and negative \emph{odd} parameters are equal and $L'=\pm P(2,-2,\dots,2,-2)$. Then $K$ can be isotoped to be of Type~3.
            \end{enumerate}
        \item The pretzel knot $K$ falls into one of the following subcases:
        \begin{enumerate}[wide=\dimexpr\parindent+1em+\labelsep\relax, leftmargin=* ]
           \item[\emph{Type 3A:}] The numbers of positive and negative parameters differ. Then $K$ fibers if and only if the difference is two. 
           \item[\emph{Type 3B:}]  The numbers of positive and negative parameters are equal and $L'\neq\pm P(2,-2,\dots,2,-2)$. Then $K$ fibers if and only if the auxiliary link is $L'=\pm P(2,-2,\dots,2,-2,-2)$.
            \item[\emph{Type 3C:}] Otherwise, $K$ is fibered if and only if there is a unique parameter of minimal absolute value.
        \end{enumerate}
     \end{enumerate}
\end{thm}
Because of the importance of the order of the parameters defining a pretzel knot, we want to stress that fiberedness is a subtle property when we consider re-ordering of parameters. For example $P(1,5,7,-5,-2)$ is a Type 2A fibered pretzel knot and all its mutants, that is, all the different Type 2 pretzel knots obtained by reordering the parameters, for example $P(1,-5,7,5,-2)$, are still fibered. On the other hand, if we consider a Type~2B fibered pretzel knot, such as $P(3,-7,5,-5,8)$, the order of the parameters is crucial: the mutant $P(3,5,-7,-5,8)$ is not fibered. 

In this article, we aim to establish which \emph{fibered} pretzel knots are ribbon. The general question of understanding which pretzel knots are slice and/or ribbon is very far from settled. Many of the known results use as a starting point the fact that the double cover of $S^3$ branched over a slice knot $K$, call it $Y_K$, bounds a rational homology ball. If one can show that $Y_K$ does not bound such a ball, then $K$ is not slice. When dealing with pretzel knots, this approach has the following interesting feature: if we start with a pretzel knot $K=P(p_1,\dots,p_n)$ and $Y_K$ does not bound a rational homology ball, we get for free that none of the mutants of $K$ is slice either. That is, we conclude that no matter what order we choose for the set of parameters $\{p_1,\dots,p_n\}$, a pretzel knot with those parameters is not slice. On the other hand, if we show that $Y_K$ does bound a rational homology ball, and say we manage to exhibit a slice disk for $K=P(p_1,\dots,p_n)$, then all the other pretzel knots defined with the same set of parameters but a different order might not be slice but this particular obstruction derived from the double cover will vanish. There is one exception to all this: 3-stranded pretzel knots are determined by their double covers, since cyclic permutation and reversal of order account for all possible orderings of three parameters. 

The first systematic study of the slice ribbon conjecture for pretzel knots was carried out by Greene and Jabuka \cite{GJ} who dealt with 3-stranded \emph{Type 1} pretzel knots. They show that if no parameter is  unitary, then such a pretzel is ribbon if and only if it is of the form $P(p,-p,q)$; if there are unitary parameters, then the pretzel knot is actually a 2-bridge knot and the slice ribbon conjecture has been proven for this class of knots in \cite{lisca}.

Chronologically, the next big step towards classifying slice pretzel knots was carried out in \cite{pretzel}, where it was shown that
\begin{thm}
    Let $K$ be a non-exceptional, non-unitary slice pretzel knot. Then, up to reordering of parameters,
    
    \begin{itemize}
        \item if K is Type 2, then  $K=P\{p_0,p_1,-p_1,p_2,-p_2,\dots,p_t,-p_t\}$, while
        \item if K is Type 3, then $K=P\{p_1,-p_1\pm1,p_2,-p_2,\dots,p_t,-p_t\}$.  
    \end{itemize}
    
\end{thm}

Here an \emph{exceptional} pretzel knot is a Type 2 knot whose parameters consist of pairs of numbers of the same absolute value and opposite signs along with $(a,-a-2,-\frac{(a+1)^2}{2})$ with $a\equiv 1,97\mod(120)$ \cite{pretzel,KimLeeSong}. While these are conjecturally not even algebraically slice, and on a case by case basis it is easy to establish that they are not, they have resisted all attempts so far to establish this statement for all parameters $a$.

The case of 5-stranded \emph{Type 1} pretzel knots was addressed in \cite{bryant}. Particularly, Corollary~2 in that article tells us that in the absence of unitary parameters, if these pretzel knots are slice then the defining set of parameters is of the form $\{p_0,p_1,-p_1,p_2,-p_2\}$. At the time of writing, beyond 5-strands, the ribbon status of Type 1 pretzel knots was wide open.

Notice the constant themes:
\begin{itemize}
      \item When we have more than three strands, classifications of ribbon pretzel knots are generally stated only up to mutation.
    \item In the absence of unitary parameters, so far the evidence shows that ribbon pretzel knots follow a clear pattern of `paired up parameters' of opposite sign (with the particular tweaks depending on the parity of both the parameters and the number of parameters).
\end{itemize}

With the currently available techniques, dealing with mutation is difficult. One of the frequently used tools to distinguish mutants are twisted Alexander polynomials. These, while effective, involve intricate computations which do not play well with infinite families. In \cite{HeraldKirkLiv} this technique was used to show that the 24 distinct oriented mutants of $P(3, 7, 9, 11, 15)$ are mutually different in the topological concordance group. Directly relevant to our exposition here is the result of Miller \cite[Theorem~1.4]{miller} who is able to show that infinitely many 4-stranded pretzels with `the right parameters but the wrong order' are not slice. 

The results in the literature pinning down the possible sets of parameters defining ribbon pretzel knots do say something about the order. In fact, \cite[Proposition~2.1]{pretzel} presents a very simple algorithm describing ribbon disks for pretzels with paired up parameters that satisfy an adjacency condition. For example, the algorithm applied to $P(3,-3,5,-5,7)$ yields a ribbon disk while it does not for the mutant $P(3,5,-3,-5,7)$. It was in fact shown in \cite{HeraldKirkLiv} that this last pretzel knot is not slice.  Conjecturally \cite[Conjecture~2.4]{pretzel}, only the pretzel knots, without unitary parameters, for which the aforementioned algorithm yields a disk, are ribbon. In this article, we will call such pretzel knots \emph{detectably ribbon} (Cf. `simple ribbon' in \cite{bryant}).

We close this background section with a word about unitary parameters. In the following sections we will show that in their presence, ribbon pretzel knots do not exactly follow the pattern of `paired up parameters' (see for example Propositions~\ref{p:odd} and~\ref{p:3A}). Moreover, the bands describing the ribbon disks are not particularly simple in the standard pretzel projection of these knots (see Figures~\ref{f:bands} and~\ref{f:band3A1}). These examples show that the ``Pretzel Slice-Ribbon Conjecture" in \cite[Section~2]{bryant} needs to be amended to exclude the case of unitary parameters.

\section{Pretzel knots, graphs and Donaldson's obstruction}\label{s:lattice}

\subsection{Negative graphs.}
 Given a pretzel link $L=P(p_1,...,p_n)$, we start this section describing how to associate to $L$ a star-shaped graph $\Gamma_L$ which encodes the plumbing instructions to build a 4-manifold $X_{\Gamma_L}$, whose boundary $Y_L$ is the double cover of $S^3$ branched over $L$ (details of the plumbing construction can be found for example in \cite{gompf20234}). In what follows, when $L$ is clear, we will drop it from the notation. For a given $L$, there are many graphs $\Gamma$ with the described properties. Among them, at most one will encode a negative definite $X_\Gamma$ (and precisely one if $L$ is a knot. This is the standard form defined in \cite[Theorem~5.2]{NeuRay}). As explained further below, this negative definite manifold will play a crucial role in our study of sliceness of pretzel knots. 
 
 The first step is to associate to $L=P([1^{\pm d}],p_{d+1},\dots,p_n)$ the following plumbing graph, with a central vertex of weight $\mp d$ connected to $n-d$ vertices, each of weight $p_i$.
 \[
  \begin{tikzpicture}[xscale=1.3,yscale=-0.6]
    \node (A0_4) at (4.25, 3) {$p_n$};
    \node (A1_4) at (4, 3) {$\bullet$};
    \node (A2_2) at (1.65, 3) {$p_{d+1}$};
    \node (A2_2) at (1.65, 4) {$p_{d+2}$};
    \node (A2_3) at (3, 2.5) {$\mp d$};
    \node (A3_1) at (2, 4) {$\bullet$};
    \node (A3_2) at (2, 3) {$\bullet$};
    \node (A3_3) at (3, 3) {$\bullet$};
    \node (A4_4) at (4.45, 4) {$p_{n-1}$};
    \node (A4) at (4, 4) {$\bullet$};
     \path (A3_3) edge [-] node [auto] {$\scriptstyle{}$} (A4);
     \path (A3_3) edge [-] node [auto] {$\scriptstyle{}$} (A3_1);
     \path (A3_3) edge [-] node [auto] {$\scriptstyle{}$} (A1_4);
     \path (A3_3) edge [-] node [auto] {$\scriptstyle{}$} (A3_2);
     \draw [loosely dotted] (A3_1) .. controls (3,5)  .. (A4);
   \end{tikzpicture}
  \]
 %
%
 The manifold $Y_L$ obtained from this graph is a Seifert fibered space and it will bound a negative definite manifold if its \emph{Euler number}
 $$
 e(Y)=\mp d-\sum_{i=d+1}^n\frac{1}{p_i}
 $$
 is negative \cite[Theorem~5.2]{NeuRay}. If $e(Y)$ is positive, then $e(-Y)$ will be negative. The manifold $-Y$ is the double cover of $S^3$ branched over the mirror of $L$, which is the pretzel knot with parameters opposite to those of $L$. Conveniently, the properties we are interested in in this article, fiberedness and sliceness/ribbonness of knots, are invariant under mirror imaging, so, up to considering the mirror of $L$, we can always find a negative 4-manifold $X$ to work with. This negative definite manifold will be constructed as follows:
 \begin{enumerate}
     \item If $e(Y)$ is negative, then we obtain a graph $\Gamma$ defining a negative definite $X_\Gamma$ from the one above by substituting every vertex with positive weight $p_i>2$ with a linear graph of $p_i-1$ vertices of weight $-2$ connected to the central vertex. For every such substitution, we subtract $1$ from the weight of the central vertex.
     \item If $e(Y)$ is positive, we change all the signs of the weights in the graph above and then proceed as in the previous point.
 \end{enumerate}
This procedure is just an algorithmic description of a series of blow-ups and blow-downs which do not change $Y_L$ and yield a negative definite plumbed 4-manifold $X_\Gamma$ (see \cite{NeuRay} for a proof). For example, starting with the pretzel knot $L=P(-1,-1,2,3,-5)=P([1^{-2}],2,3,-5)$, we compute the associated Euler number
$$
e(Y_L)=2-\frac{1}{2}-\frac{1}{3}+\frac{1}{5}>0,
$$
and then follow the steps described in (2) above which are illustrated with the following graphs, yielding the rightmost one, which is the one producing a negative definite 4-manifold via plumbing.
 \[
  \begin{tikzpicture}[xscale=0.9,yscale=0.3]
    \node (A1) at (0,0) {$\bullet$};
    \node (A2) at (1,0) {$\bullet$};
    \node (A3) at (2,0) {$\bullet$};
    \node (A4) at (4,0) {$\bullet$};
    \node (A5) at (5,0) {$\bullet$};
    \node (A6) at (6,0) {$\bullet$};
    \node (A7) at (8,0) {$\bullet$};
    \node (A8) at (9,0) {$\bullet$};
    \node (A9) at (10,0) {$\bullet$};
    \node (A10) at (11,0) {$\bullet$};
    \node (A11) at (12,0) {$\bullet$};
    \node (A12) at (13,0) {$\bullet$};

    \node (W1) at (0,1) {$2$};
    \node (W2) at (1,1) {$2$};
    \node (W3) at (2,1) {$-5$};
    \node (W4) at (4,1) {$-2$};
    \node (W5) at (5,1) {$-2$};
    \node (W6) at (6,1) {$5$};
    \node (W7) at (8,1) {$-2$};
    \node (W8) at (9,1) {$-3$};
    \node (W9) at (10,1) {$-2$};
    \node (W10) at (11,1) {$-2$};
    \node (W11) at (12,1) {$-2$};
    \node (W12) at (13,1) {$-2$};

    \node (B1) at (1,-2.3) {$\bullet$};
    \node (B2) at (5,-2.3) {$\bullet$};
    \node (B3) at (9,-2.3) {$\bullet$};

    \node (L1) at (1.4,-2.3) {$3$};
    \node (L2) at (5.4,-2.3) {$-3$};
    \node (L3) at (9.4,-2.3) {$-3$};
   
    \path (A1) edge [-] node [auto] {$\scriptstyle{}$} (A2);
    \path (A2) edge [-] node [auto] {$\scriptstyle{}$} (A3);
    \draw [->](2.7,0) --  (3.3,0);
    \path (A4) edge [-] node [auto] {$\scriptstyle{}$} (A5);
    \path (A5) edge [-] node [auto] {$\scriptstyle{}$} (A6);
    \draw [->](6.7,0) --  (7.3,0);
    \path (A7) edge [-] node [auto] {$\scriptstyle{}$} (A8);
    \path (A8) edge [-] node [auto] {$\scriptstyle{}$} (A9);
    \path (A9) edge [-] node [auto] {$\scriptstyle{}$} (A10);
    \path (A10) edge [-] node [auto] {$\scriptstyle{}$} (A11);
    \path (A11) edge [-] node [auto] {$\scriptstyle{}$} (A12);

    \path (A2) edge [-] node [auto] {$\scriptstyle{}$} (B1);
    \path (A5) edge [-] node [auto] {$\scriptstyle{}$} (B2);
    \path (A8) edge [-] node [auto] {$\scriptstyle{}$} (B3);
    
\end{tikzpicture}
  \]


From the plumbing construction, it is immediate to see that $b_2(X_\Gamma)$ is the number of vertices in the graph. Moreover, in the natural basis of $H_2(X_\Gamma)$ given by the spheres represented by the vertices of the graph, we have that the intersection form of $X_\Gamma$ coincides with the incidence matrix of the graph, which we will denote $Q_\Gamma$. We will say that the graph $\Gamma$ is negative definite when $Q_\Gamma$ has this property.

\subsection{Donaldson's obstruction.}
The reason for constructing a negative definite four manifold bounded by the double cover of $S^3$ branched over a pretzel knot is that we want to use Donaldson's Theorem~A \cite{Donaldson}  to obstruct sliceness. This obstruction has been described many times in the literature and, in the particular form we will be using it, it was introduced in \cite{lisca}. If a knot $K$ is slice (or a link is $\chi$-slice), then its double branched cover $Y_K$ bounds a rational homology ball $W$. If $Y_K$ also bounds a negative definite manifold $X$, we can construct a closed negative definite  4-manifold as $C:=X\cup_{Y_K}-W$. Donaldson's Theorem~A tells us then that the intersection form $Q_C$ is diagonalizable. In other words, the lattice $(H_2(C),Q_C)$ is lattice isomorphic to $(\Z^{b_2(C)},-\mathrm{Id})$. Moreover, since $W$ is a rational homology ball, a straightforward application of the Mayer-Vietoris sequence with rational coefficients shows that $H_2(X)$ injects into $X_2(C)$ and, in particular $b_2(X)=b_2(C)$. This in turn implies that there is a \emph{lattice embedding} of $(\Z^{b_2(X)},Q_X)$ into $(\Z^{b_2(C)},-\mathrm{Id})$, that is, there is an injection $\iota:\Z^{b_2(X)}\rightarrow\Z^{b_2(C)}$ such that $Q_X(a,b)=-\mathrm{Id}(\iota(a),\iota(b))$. The existence of this injection is derived from the existence of a slice disk for $K$. In the proofs in the next section we will show that certain families of pretzel knots are not slice by obstructing the existence of $\iota$. 

Going back to the concrete case of pretzel knots, once we have associated to a knot the negative definite graph $\Gamma$, with say $k$ vertices, Donaldson's obstruction amounts to understanding when is there an embedding of the lattice $(\Z^k,Q_\Gamma)$ into the standard negative definite lattice of the same rank. Because of the lattice vocabulary, we will often refer to the number of vertices in a given graph as the \emph{rank} of the graph. If there is an embedding $\iota$, then the image through this embedding of the homology class associated to each vertex $v_i$ can be expressed as a linear combination of the basis vectors $\{e_1,\dots,e_k\}$ spanning the standard negative definite lattice (i.e.\ $e_i\cdot e_j=-\delta_{ij}$ for all indices). Many times, instead of writing $\iota(v_i)=\sum\alpha_j^ie_j$, by a slight abuse of notation, we will omit the $\iota$. When we say \emph{a graph $\Gamma$ embeds} (resp.\ \emph{$\Gamma$ does not embed}), we always mean that the lattice associated to it embeds (resp.\ does not embed) into the standard negative lattice of the same rank. 

To determine the existence or impossibility of an embedding, we will manipulate the graphs and potential embeddings in several ways. There is some associated technical jargon, mostly intuitive expressions, which are generally abuse of notation. Among them:
\begin{itemize}
    \item Many times we claim that a vertex or a subgraph \emph{has a particular embedding}. This is to be interpreted up to appropriate automorphisms of $(\Z^k,-\mathrm{Id})$, which include relabeling of the basis vectors and sign changes. 
    \item When we say that a basis vector $e_s$ \emph{appears} in the embedding of a subgraph, we mean that there is a vertex $v$ in that subgraph for which in the expression $v=\sum\alpha_j e_j$ the coefficient $\alpha_s$ is not zero. 
    \item Some times we will look at a sublattice spanned by a subset $S$ of the $e_i$ and consider the \emph{associated graph} (which, importantly, might not be a subgraph of $\Gamma$). This graph with an embedding is obtained by considering only the vertices of $\Gamma$ where the basis vectors in $S$ appear. These vertices will have a new embedding, obtained by simply deleting from their original embedding the vectors not in $S$. The edges of the new graph will be such that they reflect the pairings of the basis vectors in $S$. A simple example illustrates the idea: in the following diagram, to the left we have an embedding of a graph into the lattice ($\Z^5=\langle e_1,\dots,e_5\rangle,-\mathrm{Id})$; letting $S$ be the subset $\{e_2,e_3\}$ we can consider the associated graph, which is given to the right of the diagram.
 \[
  \begin{tikzpicture}[xscale=1.7,yscale=0.4]
    \node (A1) at (0,0) {$\bullet$};
    \node (A2) at (1,0) {$\bullet$};
    \node (A3) at (2,0) {$\bullet$};
    \node (A4) at (4,0) {$\bullet$};
    \node (A5) at (5,0) {$\bullet$};
    \node (A6) at (6,0) {$\bullet$};

    \node (W1) at (0,1) {\tiny{$e_1-e_2$}};
    \node (W2) at (1,1) {\tiny{$-e_1+e_3$}};
    \node (W3) at (2,1) {\tiny{$e_1+e_2$}};

    \node (B1) at (1,-2.3) {$\bullet$};
    \node (B2) at (1,-4.3) {$\bullet$};
    \node (B3) at (5,-2.3) {$\bullet$};

    \node (L1) at (1.4,-2.3) {\tiny{$-e_3+e_4$}};
    \node (L2) at (1.4,-4.3) {\tiny{$-e_4+e_5$}};
   
    \node (L3) at (4,1) {\tiny{$-e_2$}};
    \node (L4) at (6,1) {\tiny{$e_2$}};
    \node (L5) at (5.2,0) {\tiny{$e_3$}};
    \node (L6) at (5.2,-2.3) {\tiny{$-e_3$}};
    
    \path (A1) edge [-] node [auto] {$\scriptstyle{}$} (A2);
    \path (A2) edge [-] node [auto] {$\scriptstyle{}$} (A3);
    \draw [->,zigzag](2.5,0) --  (3.7,0);
    \path [-] (A4) edge [out= 60, in= 120] (A6);
    \path (A2) edge [-] node [auto] {$\scriptstyle{}$} (B1);
    \path (B1) edge [-] node [auto] {$\scriptstyle{}$} (B2);
    \path (A5) edge [-] node [auto] {$\scriptstyle{}$} (B3);
\end{tikzpicture}
  \]

\end{itemize}

In the proofs, we will make use several times of the following observation. A pretzel \emph{knot} has odd determinant, and, in particular, it is non-zero. The determinant of these knots coincides with the determinant of $Q_\Gamma$ so, in particular,
\begin{re}\label{r:useful}
    A graph with $n$ vertices associated to a pretzel \emph{knot} cannot embed into a standard negative lattice of rank less than $n$.
\end{re}%
\noindent Expressed in more general algebraic terms, if the determinant of $Q_\Gamma$ is non-zero then the lattice associated to $\Gamma$ cannot embed into a lattice of smaller rank.

\subsection{Signature.}
We close this section with some useful facts about the computation of the signature of pretzel knots. A formula due to Saveliev \cite[Theorem~5]{Sav00} expresses the signature of a pretzel knot $K$ as
\begin{equation}\label{e:sign}
  \sigma(K)=\mathrm{sign}(Q_{\Gamma})- Q_{\Gamma}(w,w),  
\end{equation}
where sign stands for the signature  and $w\in\mathrm{H_2}(X_{\Gamma};\Z)$ is the so-called Wu class. This class is the unique integral lift to $\mathrm{H_2}(X_{\Gamma};\Z)$, with coordinates 0 or 1 in the basis of $\mathrm{H_2}(X_{\Gamma};\Z)$ given by the vertices of the graph, of the unique solution in $\mathrm{H_2}(X_{\Gamma};\Z_2)$ of the equation
\begin{equation}\label{e:cong}
Q_\Gamma(w,x)\equiv Q_\Gamma(x,x)\mod 2\quad\quad\forall x\in\mathrm{H_2}(X_{\Gamma};\Z).
\end{equation}
There is an easy algorithm to determine the Wu set, that is the set of vertices $v_i^w$ of the graph which satisfy $w=\sum_i v_i^w$ [see for example \cite[Section 3.2]{pretzel} and \cite{NeuRay}]. An important feature that we want to note here is that, for ribbon pretzel knots and when working with a negative definite graph, we can always assume that the embedding of the Wu class in the closed manifold $X_\Gamma\cup W$ (where $W$ is a rational homology ball) is the sum of all the basis vectors of the standard negative definite lattice, in symbols, $\iota(w)=\sum_i^ke_i$. Indeed, as $\det(Y_K)=\det(Q_\Gamma)$ is odd, $\iota(H_2(X_\Gamma))\subset H_2(X_\Gamma\cup W)$ has odd index, so reduction modulo 2 is surjective; thus in this case, Equation~\eqref{e:cong} extends to all $x\in H_2(X_\Gamma\cup W)$ and via the embedding $\iota$ reads
$$
-\mathrm{Id}(\iota(w),x)\equiv -\mathrm{Id}(x,x)\mod 2\quad\quad\forall x\in\mathrm{H_2}(X_\Gamma\cup W;\Z).
$$
Considering this last congruence in particular for the classes $x=e_i$, for each standard basis vector $e_i$, we obtain that all the coefficients appearing in the expression of $\iota(w)$ need to be odd (and in particular non-zero). Moreover, since ribbon knots have signature zero, Equation~\eqref{e:sign} reads $-k=Q_\Gamma(w,w)=-\mathrm{Id}(\iota(w),\iota(w))$, which immediately implies that the coefficients need to be at most 1 in absolute value. By eventually applying some automorphisms of $(\Z^k,-\mathrm{Id})$, sending a basis element $e_i$ to $-e_i$, we can assume all the coefficients are $+1$ and therefore $\iota(w)=\sum_i^ke_i$ as claimed.

\section{Proofs}\label{s:lemmas}

This lengthy section is devoted to the proofs of the different statements in the introduction. We begin with fibered ribbon Type 1  knots, dealt with completely in Proposition~\ref{p:odd}. Once this is settled, we move on to Types~2 and~3. We distinguish the case of non-unitary parameters, for which much is known, from the one with unitary parameters, for which much more work needs to be done in order to establish the desired results. 

\begin{prop}\label{p:odd}
The pretzel knot $\pm P(1,1,1,1,-3,-3,-3)=\pm 10_{75}$ is the only Type 1 fibered ribbon pretzel knot. 
\end{prop}
\begin{proof}
From Gabai's classification (Theorem~\ref{t:Gabai}) of fibered pretzel knots, we know that Type 1 fibered pretzel knots are of the form $K=P(1,\dots,1,-3,\dots,-3)$ or $\bar K=P(-1,\dots,-1,3,\dots,3)$ with at least one unitary parameter and an odd number of parameters. Our task is to determine which among these are ribbon and it suffices to study the family of knots $K$. To this end, we will associate to $K$ a negative definite graph and we will try to embed it into the standard negative lattice of the same rank (see Sections~\ref{s:background} and~\ref{s:lattice} for notation and strategy). 

In order to determine which is the negative definite graph associated to a Type~1 fibered knot with $n$ parameters $K=([1^d],-3,\dots,-3)$  we need to first compute and understand the sign of its Euler number
$$
e(Y_K)=\frac{n-d}{3}-d.
$$
If $n<4d$, then $e(Y_K)<0$ and the negative definite graph $\Gamma_K$, with the Wu-set marked in red, is
 \[
  \begin{tikzpicture}[xscale=1.3,yscale=-0.6]
    \node (A0_4) at (4.25, 3) {$-3$};
    \node (A1_4) at (4, 3) {$\bullet$};
    \node (A2_2) at (1.65, 3) {$-3$};
    \node (A2_2) at (1.65, 4) {$-3$};
    \node (A2_3) at (3, 2.4) {$-d$};
    \node (A3_1) at (2, 4) {$\bullet$};
    \node (A3_2) at (2, 3) {$\bullet$};
    \node (A3_3) at (3, 3) {$\bullet$};
    \node (A4_4) at (4.45, 4) {$-3$};
    \node (A4) at (4, 4) {$\bullet$};
    \draw [color=red] (3,3) ellipse (.15 and -.33);
     \path (A3_3) edge [-] node [auto] {$\scriptstyle{}$} (A4);
     \path (A3_3) edge [-] node [auto] {$\scriptstyle{}$} (A3_1);
     \path (A3_3) edge [-] node [auto] {$\scriptstyle{}$} (A1_4);
     \path (A3_3) edge [-] node [auto] {$\scriptstyle{}$} (A3_2);
     \draw [loosely dotted](A3_1) .. controls (3,5)  .. (A4);
   \end{tikzpicture}
  \]
On the other hand, if $n>4d$, then $e(Y_K)>0$ and the negative definite graph $\Gamma'_K$ with its corresponding Wu-set in red is
 \[
  \begin{tikzpicture}[xscale=1.3,yscale=-0.6]
    \node (A0_4) at (4, 2.4) {$-2$};
    \node (A2_2) at (2, 2.4) {$-2$};
    \node (A2_2) at (1,2.4) {$-2$};
    \node (A2_3) at (3, 2.4) {$2d-n$};
   
   \node (A3_0) at (1, 3) {$\bullet$};
     \node (A1_4) at (4, 3) {$\bullet$};
     \node (A3_5) at (5, 3) {$\bullet$};
     \node (A3_2) at (2, 3) {$\bullet$};
    \node (A3_3) at (3, 3) {$\bullet$};
    \node (A3_1) at (2, 4) {$\bullet$};
    \node (A4) at (4, 4) {$\bullet$};
    \node (B) at (1.2, 4.5) {$\bullet$};
    \node (C) at (4.7, 4.5) {$\bullet$};
   
    \draw [color=red] (1,3) ellipse (.15 and -.33);
    \draw [color=red] (5,3) ellipse (.15 and -.33);
    \draw [color=red] (1.2, 4.5) ellipse (.15 and -.33);
    \draw [color=red] (4.7, 4.5) ellipse (.15 and -.33);
    \node at (5,2.4) {$-2$};

    \node (A4_4) at (4.3, 3.8) {$-2$};
    \node at (5, 4.5) {$-2$};
    \node at (0.8, 4.5) {$-2$};
    \node at (1.6, 3.8) {$-2$};
    
    \draw [color=red] (3,3) ellipse (.15 and -.33);
     \path (A3_3) edge [-] node [auto] {$\scriptstyle{}$} (A4);
     \path (A3_0) edge [-] node [auto] {$\scriptstyle{}$} (A3_2);
     \path (A3_2) edge [-] node [auto] {$\scriptstyle{}$} (A3_3);
     \path (A3_3) edge [-] node [auto] {$\scriptstyle{}$} (A1_4);
     \path (A1_4) edge [-] node [auto] {$\scriptstyle{}$} (A3_5);
     \path (A3_3) edge [-] node [auto] {$\scriptstyle{}$} (A3_1);
     \path (A3_3) edge [-] node [auto] {$\scriptstyle{}$} (A4);
     \path (A4) edge [-] node [auto] {$\scriptstyle{}$} (C);
     \path (A3_1) edge [-] node [auto] {$\scriptstyle{}$} (B);
     \draw [loosely dotted] (A3_1) .. controls (3,5)  .. (A4);
   \end{tikzpicture}
  \]

%
%
We assume $K$ is slice. Then $\sigma(K)=0$ and Equation~\eqref{e:sign} reads $0=-(n-d+1)+d$ for $\Gamma_K$ and $0=-(2(n-d)+1)+2(n-d)-2d+n$ for $\Gamma'_K$. The latter condition immediately yields a contradiction: $n=2d+1$ is incompatible with $n>4d$, the negative definite condition of the graph. As such, we will restrict attention to graphs of the first kind with the constraint $n+1=2d$.

Since $K$ is slice, the lattice $(\Z^{\mathrm{rk}(\Gamma)},Q_\Gamma)$ embeds into the standard negative lattice of the same rank. With the notation conventions of Section~\ref{s:lattice}, we will show that this only happens for the knot $10_{75}$. We will argue depending on the value of $m:=n-d$, which in our case can be written $m=d-1$.

Notice that for $m\in\{1,2\}$ the manifold $Y_K$ is a lens space. In these two cases, while it is immediate to show by hand that there is no embedding, the non-existence follows from work of Lisca \cite{lisca}. 

When $m=3$, our constraint implies $d=4$, corresponding to the knot $10_{75}$ and an embedding into the diagonal negative definite lattice of rank 4 is as follows. The central vertex $v_c$, which is the Wu class, has embedding $v_c=e_1+e_2+e_3+e_4$. The three weight 3 vertices have embeddings $v_1=-e_1+e_4-e_2$, $v_2=-e_1+e_3-e_4$ and $v_3=-e_1+e_2-e_3$. The manifold $Y_K$ and this embedding are well known and studied (see for example \cite[Figure~23]{Stipsicz}).

If $m> 3$, we claim there is no embedding. We are trying to embed $m$ weight $-3$ orthogonal vertices and the central vertex into $(\Z^{m+1},-\mathrm{Id})$. We will argue that the graphs we are considering embed only in standard lattices of rank strictly larger than $m+1$.  Having embedded the central vertex $v_c=e_1+e_2+\cdots+e_{m+1}$, the only embedding of a linked weight $-3$ vertex is of the form $v=-e_i+e_j-e_k$. It is immediate to see that the minimum number of basis vectors needed to embed:
\begin{itemize}
    \item one weight $-3$ vertex is three, $v_1=-e_1+e_2-e_3$; 
    \item two orthogonal weight $-3$ vertices is four,  $v_1=-e_1+e_2-e_3$ and $v_2=-e_1-e_2+e_4$;
    \item three orthogonal weight $-3$ vertices is four, $v_1=-e_1+e_2-e_3$, $v_2=-e_1-e_2+e_4$, and $v_3=-e_1+e_3-e_4$.
\end{itemize}
Notice that, once the embedding of a set of three orthogonal weight $-3$ vertices using four basis vectors is fixed, any further weight $-3$ orthogonal vertex linked to the central vertex cannot share any basis vectors with these three. Similarly, once a set of $3l$ orthogonal weight $-3$ vertices using $4l$ basis vectors is fixed, any further weight $-3$ orthogonal vertex cannot share any basis vectors with these.  It follows easily that when embedding $n$ weight $-3$ orthogonal vertices into a rank $k$ lattice, the difference $k-n$ is minimal when $n = 0 \mod 3$, in which case $k-n = n/3$, so in particular is greater than 1 whenever $n>3$.

To finish the proof, we need to show that the knot $10_{75}$ is indeed ribbon. While this information is available in KnotInfo \cite{knotinfo}, for completeness and comparison we include the band moves that turn this knot into a disjoint union of two unknots, both in the minimal alternating projection of $10_{75}$ and in the standard projection of this knot in pretzel form. See Figure~\ref{f:bands}.
\end{proof}

\begin{figure}
\centering
\includegraphics[width=0.85\textwidth]{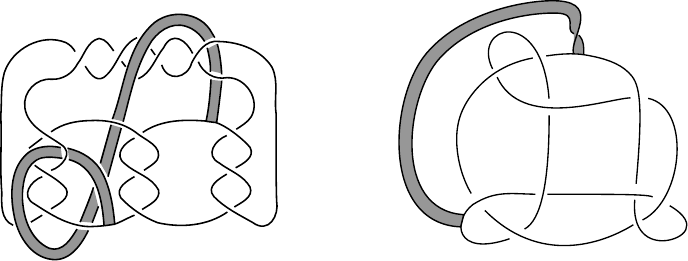}
\caption{To the left, a pretzel projection of $K=P([1^4],-3,-3,-3)$ with a band describing a move to two unlinked unknots. To the right, the same band after isotoping $K$ to an alternating projection that allows us to identify it as the knot $10_{75}$.}
\label{f:bands}
\end{figure}

\begin{re}
Notice that we have actually proved that  $P([1^4],-3,-3,-3)$ is the only Type 1 fibered \emph{slice} pretzel knot. This slice knot happens to be ribbon, which could potentially be a more restrictive condition.
\end{re}

We move on to analyze pretzel knots of Types 2 and 3. From the point of view of sliceness/ribonness these were thoroughly studied in \cite{pretzel}, but in that article the parameters $\pm1$ were excluded from the discussion. We divide the rest of this section in two. The first subsection will deal with the case of no $\pm1$ parameters and collect known facts about this class. The second subsection will deal with the remaining cases which need to be studied from scratch from the sliceness perspective.

\subsection{Pretzel knots of Types 2 and 3 with no unitary parameters.}
We are not aware of anywhere in the literature where the results from  Gabai's classification of pretzel knots of Types 2 and 3, Theorem~\ref{t:Gabai}, have been contrasted with \cite[Theorem~1.1]{pretzel}. We state here what we obtain when comparing these two classification results.
\begin{enumerate}
    \item Type 2A fibered pretzel knots are never slice. Indeed, all Type 2A pretzel knots are of the form \textbf{(p2)} or \textbf{(p3)} in \cite{pretzel}. In that article, Lemmas~5.4 and~5.5 establish, under the only assumption of these knots being slice, that the number of positive and negative parameters must differ exactly by one, whereas Type 2A requires this difference to be two.
    \item All Type 2B fibered pretzel knots have ribbon mutants. Stated otherwise, the set of parameters defining these pretzel knots can be reordered so as to obtain a fibered ribbon pretzel knot. 
    \item Type 3A fibered pretzel knots cannot be slice. Indeed, these knots have an unequal number of positive and negative parameters, which excludes them from being slice by \cite[Lemma~5.3]{pretzel}. (Note that  Type 3 pretzel knots are of the form \textbf{(p1)}.) 
    
    Type 3B pretzel knots with no unitary parameters are never fibered: since the number of positive and negative parameters is equal, $L'$ has an even number of entries; however, a Type 3B fibered pretzel knot has an associated $L'$ with an odd number of entries.

     Finally, Type 3C pretzel knots are fibered if and only if there is a unique parameter of minimal absolute value. From the ribbon point of view and up to reordering, Type~3 ribbon pretzel knots are determined by a set of parameters which consists of a certain number of paired up parameters of the same absolute value and opposite sign, and a last unordered pair of the form $\pm(k,-k-1)$ with $k>1$. It follows that in order to be Type~3C fibered and ribbon, the parameter $k$ needs to be the unique minimal positive term. For example $P(-3,3,-3,4)$ is ribbon but not fibered, while $P(-3,3,-3,2)$ is ribbon and fibered.

\end{enumerate}

We summarise these observations in the following proposition, remarking that the first family corresponds to Type~2B knots and the second one to Type~3C in Gabai's classification.
\begin{prop}\label{prop_nonunitary}
    Let $K$ be a prime non-exceptional pretzel knot with no unitary parameters. Then, $K$ is fibered and ribbon up to reordering of parameters if and only if either 
    \begin{enumerate}
        \item $K=\pm P(q_1,\ldots,q_r,-q_1,\ldots,-q_r,k)$ with $q_i \geq 3$, $k$ even, or
         \item $K = \pm P(k,-k-1, q_1,\ldots,q_r,-q_1,\ldots,-q_r$) with $q_i \geq3$, $r\geq 0$, and $1 < k < q_i$ for all $i$.
    \end{enumerate}
\end{prop}


A final immediate remark regarding the order of the parameters:
When the order  is relevant to fiberedness, the condition is generally that we need parameters alternating sign; when we are dealing with ribonness, we need to be able to pair up the parameters of opposite sign and equal absolute value. In the introduction we have showcased this fact with the knot $P(3,-3,5,-5,7,-7,4)$ and its mutants.


\subsection{Pretzel knots of Types 2 and 3 with unitary parameters.}
In what follows, we analyze the sliceness of pretzel knots of Types 2 and 3 with at least one unitary parameter. Since our goal is to compare this set of knots with fibered pretzel knots, we divide our study in the same families as Gabai.

The first family is Type 2A, whose fibered members, up to mirroring, are of the form $K=P([1^d],q_1,\dots,q_s,2)$ where $|q_i|\geq 3$ odd and $d+s$ is even. We are interested in the set of parameters $\{q_1,\dots,q_s\}$ which we will separate into a set of positive parameters $\{p_1,\dots,p_t\}$ and a set of negative ones $\{n_1,\dots,n_r\}$. The fiberedness condition implies that $d+t-r=\pm2$.

\begin{prop}\label{p:2A}
Type 2A fibered pretzel knots with unitary parameters are not slice.
\end{prop}
\begin{figure}
  \begin{tikzpicture}[xscale=1.3,yscale=-0.6]
    \node (A0_4) at (4.2, 3.8) {$2$};
    \node (A1_4) at (4, 4) {$\bullet$};
    \node (A2_2) at (1.7, 3.8) {$p_1$};
    \node (A2_2) at (1.7, 5) {$p_t$};
    \node (A2_3) at (3, 3.4) {$-d$};
    \node (A3_1) at (2, 5) {$\bullet$};
    \node (A3_2) at (2, 4) {$\bullet$};
    \node (A3_3) at (3, 4) {$\bullet$};
    \node (B) at (3.2,5.5) {$\bullet$};
    \node at (3.2,5.9) {$n_1$};
    \node at (4.3, 5) {$n_r$};
    \node (A4) at (4, 5) {$\bullet$};
     \path (A3_3) edge [-] node [auto] {$\scriptstyle{}$} (A4);
     \path (A3_3) edge [-] node [auto] {$\scriptstyle{}$} (A3_1);
     \path (A3_3) edge [-] node [auto] {$\scriptstyle{}$} (A1_4);
     \path (A3_3) edge [-] node [auto] {$\scriptstyle{}$} (A3_2);
     \path (A3_3) edge [-] node [auto] {$\scriptstyle{}$} (B);
     \draw [dotted] (A3_2) -- (A3_1);
     \draw [dotted] (B) .. controls (3.6,5.5) .. (A4);
     
    \path[draw,thin,->] 
    (4.3,2.6) -- ++(0,-1.2cm)  -- (5.4,1.4);
    \node at (5.8,1.4) {$\Gamma_1:$};
    \node at (4.6,1) {\small{$e(Y_K)<0$}};
    
    \path[draw,thin,->] 
    (4.3,6) -- ++(0,1cm)  -- (5.4,7);
    \node at (5.8,7) {$\Gamma_2:$};
    \node at (4.6,7.5) {\small{$e(Y_K)>0$}};
   
          \node (B1) at (6.8, 0) {$\bullet$};
          \node at (6.8,-0.4) {$-2$};
         \node (B2) at (8.1, 0) {$\bullet$};
         \draw [dotted] (B1) -- (B2);
         \node at (8.1,-0.4) {$-2$};
         \draw[decorate,decoration={brace,amplitude=5pt,raise=.5cm},yshift=0pt] (6.7, 0) -- (8.2, 0) node [midway,yshift=.9cm]{$p_{1}\!-\!1$};
    \node (B3) at (9, 0) {$\bullet$};
    \draw [-] (B2) -- (B3);
    \node at (9,-0.4) {$-d\!-\!t\!-\!\!1$};
     \node (B4) at (10, 0) {$\bullet$};
     \draw [color=red] (10,0) ellipse (.15 and -.33);
     \draw [-] (B3) -- (B4);
     \node at (10.3,-0.4) {$-2$};

      \node (C1) at (7, 3) {$\bullet$};
          \node at (6.9,2.6) {$-2$};
         \node (C2) at (8.2, 2) {$\bullet$};
         \draw [dotted] (C1) -- (C2);
         \node at (7.9,1.8) {$-2$};
          \draw[loosely dotted] (6.8, 0.4) .. controls (6.7,1.4) ..  (7, 2.3);
         \draw[decorate,decoration={brace,amplitude=5pt,raise=.2cm,mirror},yshift=0pt] (7, 3) -- (8.2, 2) node [midway,sloped,rotate=60,below,yshift=-.4cm]{$p_{t}\!-\!1$};
         \draw [-] (C2) -- (B3);
        \node (D1) at (9.3, 2.2) {$\bullet$};
        \draw [color=red] (9.3,2.2) ellipse (.15 and -.33);
          \node at (9.7,2.4) {$n_1$};
           \draw  (D1) -- (B3);
         \node (D2) at (10, 1) {$\bullet$};
         \draw [color=red] (10,1) ellipse (.15 and -.33);
           \node at (10.4,1.1) {$n_r$};
        \draw[loosely dotted] (9.5, 2.1) .. controls (9.7,2.1) ..  (9.9, 1.3);
         \draw (D2) -- (B3);
        \node (E1) at (6.8, 6) {$\bullet$};
        \draw [color=red] (6.8,6) ellipse (.15 and -.33);
        \node at (6.4, 5.7) {$-p_1$};
        \node (E2) at (7.8, 6) {$\bullet$};
        \node at (7.8, 5.6) {$d\!-\!r$};
        \draw [-] (E1) -- (E2);
        \node (E3) at (8.8, 6) {$\bullet$};
        \node at (8.8, 5.6) {$-2$};
        \draw [-] (E2) -- (E3);
        \node (E4) at (6.9, 7.5) {$\bullet$};
        \draw [color=red] (6.9,7.5) ellipse (.15 and -.33);
        \draw [-] (E4) -- (E2);
        \draw[loosely dotted] (6.7, 6.4) .. controls (6.7,6.8) ..  (6.8, 7.2);
        \node at (6.5, 7.6) {$-p_t$};
        \node (F1) at (8.8, 7) {$\bullet$};
        \node at (8.8, 7.6) {$-2$};
        \node (F2) at (10, 7.9) {$\bullet$};
        \node at (10, 8.4) {$-2$};
        \draw [dotted] (F1) -- (F2);
         \draw [-] (F1) -- (E2);
         \draw[decorate,decoration={brace,amplitude=5pt,raise=.2cm},yshift=0pt] (8.8, 7) -- (10, 7.9) node [midway,sloped,rotate=-54,above,yshift=.35cm]{\tiny{$|n_{r}|\!-\!1$}};

         \node (D1) at (7.6, 8) {$\bullet$};
         \node at (7.9, 7.8) {$-2$};
         \node (D2) at (8.8, 9) {$\bullet$};
        \node at (9.1, 8.9) {$-2$};
        \draw [dotted] (D1) -- (D2);
         \draw [-] (D1) -- (E2);
         \draw[decorate,decoration={brace,amplitude=5pt,raise=.2cm,mirror},yshift=0pt] (7.6, 8) -- (8.8, 9) node [midway,sloped,rotate=-62,above,yshift=-.9cm]{\tiny{$|n_{1}|\!-\!1$}};

         \draw[loosely dotted] (7.8, 7) .. controls (8,7.1) ..  (8.3, 6.5);
   \end{tikzpicture}
\caption{Two possible negative definite graphs associated to a Type 2 pretzel knot $K$ with $d$ unitary parameters, depending on the value of $e(Y_K)$. The Wu class in each case is depicted in red.}
\label{f:2Aones}
\end{figure}
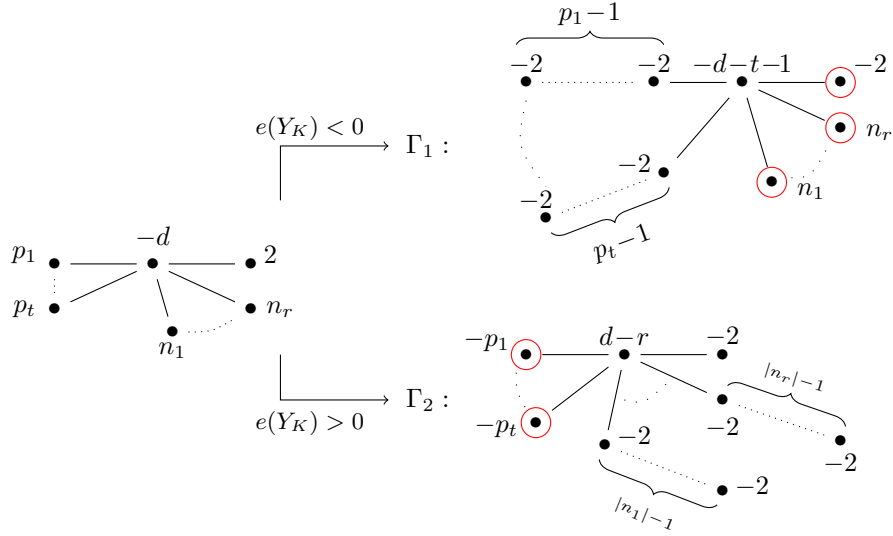
%
%
\begin{proof}
We will argue by contradiction: assuming such a knot $K$ is slice, we have that the Seifert manifold $Y_K$ bounds a rational homology ball, which in turn implies the existence of an embedding of one of the two possible negative graphs, $\Gamma_1$ and $\Gamma_2$, associated to $K$ which are depicted in \autoref{f:2Aones}. 

The existence of an embedding $\iota$ of $(\Gamma_1,Q_{\Gamma_1})$ into the standard negative lattice of the same rank, and the particular form of the embedding of the Wu class, imply that the central vertex of $\Gamma_1$, call it $v_c$, has weight at most $-r-1$ (each of the $r+1$ vertices in the Wu set shares at least one basis vector with the central vertex, and these are all different). On the other hand, since $K$ is fibered,  $d+t-r=\pm2$ and therefore, the weight of the central vertex satisfies $v_c\cdot v_c=-d-t-1=-r-1\mp2$. As this quantity is bounded above by $-r-1$, the sign of 2 is fixed, giving $v_c\cdot v_c=-r-3$, and $d+t-r=2$. 
Moreover, Equation~\eqref{e:sign} in this context reads:
\begin{equation}\label{e:rank}
 \sum_{i=1}^r|n_i|+2=\sum_{i=1}^t p_i - t +r+2 \Leftrightarrow \sum_{i=1}^r|n_i|+2=\sum_{i=1}^t p_i +d.
\end{equation}

Labelling the vertex with weight $n_i$ by $v_i$, and $w'$ the unique vertex in the Wu set with weight $-2$, we can fix the embedding of the Wu class as the sum of $v_i=e^i_1+\dots+e^i_{|n_i|}$ and $w'=f_1+f_2$, where the standard basis of the negative diagonal lattice is given by $\{e^1_1,\dots,e^r_{|n_r|},f_1,f_2\}$ and, via Equation~\eqref{e:rank}, the rank of $\Gamma_1$ is $\sum|n_i|+2$. We can moreover fix the embedding $v_c=-e^1_1-e^2_1-\dots-e^r_1-f_1+e^1_2-e^1_3$: we need the first $r+1$ basis vectors with these signs to connect the central vertex to the $v_i$ and to $w'$; the last two basis vectors yield the correct weight for $v_c$. Since $|n_i|\geq 3$, up to relabelling the vertices, we can pick the two basis vectors from the embedding of $v_1$. The only real choice we have made in the embedding of $v_c$ is the sign of $e^1_2$ (which forces the sign of $e^1_3$). Finally,  we label the $-2$-weight vertices arising from blowing up/down the $p_j$-weight vertex of $\Gamma_K$ by $u^j_1, \ldots , u^j_{p_j-1}$, where sub-indices increase moving away from $v_c$, denote by $U^j$ the union of these vertices, which we will refer to as a $-2$-chain, and define $U = \cup_j U^j$.
 
In the following, we will make extensive use of the following:
\begin{ob} \label{ob_uniquely_used_pair}
    Suppose there is a pair $e^j_l, e^j_{l+1}$ such that neither element appears in $\iota(U)$. Then by simply deleting this pair we obtain an embedding of a new graph $\Gamma'$ which differs from $\Gamma_1$ only in that the weight of the vertex $v_j$ has increased by 2, the weight of $v_c$ is either unchanged or increased by 2, and again corresponds to a pretzel knot. As the vertex set is unchanged, yet $\Gamma'$ embeds in a lattice of rank rk$(\Gamma)-2$, we arrive at a contradiction of Remark 3.1.
\end{ob} 

On the other hand, if $f_i$, $i \in\{1,2\}$ were to appear in the embedding of some $u\in U$, then by orthogonality to $w'$ we would have $\iota(u)=f_1+f_2$ (up to sign). This would then imply that $v_c\cdot u = 1$, so $u=u^l_1$ for some $l$. As however $p_l\geq3$ for all $l$, there is a vertex $u^l_2$ whose embedding contains only one $f_i$, and thus links to $w'$, a contradiction.

 Combining these observations, we see that each $\iota(u^j_1)$ must contain a vector of the form $e^k_*$ and not $f_1$. If this vector is $e^k_1$ for some $k>1$, then the structure of the $-2$-chain immediately implies that (up to relabelling) $\iota(u^j_i)=e^k_i-e^k_{i+1}$ for each $i$, and so by Observation \ref{ob_uniquely_used_pair}, $|n_k| = p_j$ (recall that all parameters $n$ and $p$ are odd). Moreover, orthogonality implies straight away that $\iota(u^{s}_1)$ cannot contain $e_1^k$ for $s\neq j$. It follows that for each $k>1$, there is a unique $U^j$ whose embedding contains exactly the vectors in $\iota(v_k)$. So far we have established that the embedding of the graph $\Gamma_1$ is as in Figure~\ref{f:partial1}.
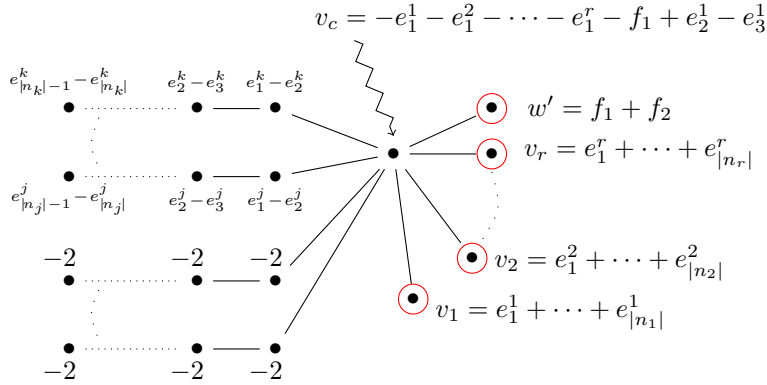
\begin{figure}[h]
   \begin{tikzpicture}[xscale=1.3,yscale=-0.6]
       \node (B1) at (5.7, -1) {$\bullet$};
          \node at (5.7,-1.6) {\tiny{${e_{|\!n_k\!|-1}^k\!-\!e_{|\!n_k\!|}^k}$}};
          \node (X1) at (7, -1) {$\bullet$};
          \draw [dotted] (B1) -- (X1);
          \node at (7,-1.6) {\tiny{$e_2^k\!-\!e_3^k$}};
         \node (B2) at (7.8, -1) {$\bullet$};
         \draw [-] (X1) -- (B2);
         \node at (7.8,-1.6) {\tiny{$e_1^k\!-\!e_2^k$}};
       \node (K1) at (5.7, 0.5) {$\bullet$};
          \node at (5.7,1) {\tiny{${e_{|\!n_j\!|-1}^j\!-\!e_{|\!n_j\!|}^j}$}};
          \node (Y1) at (7, 0.5) {$\bullet$};
          \draw [dotted] (K1) -- (Y1);
          \node at (7,1) {\tiny{$e_2^j\!-\!e_3^j$}};
         \node (K2) at (7.8, 0.5) {$\bullet$};
         \draw [-] (Y1) -- (K2);
         \node at (7.8,1) {\tiny{$e_1^j\!-\!e_2^j$}};
    
    \draw[loosely dotted] (6, -0.8) .. controls (5.9,-0.2) ..  (6, 0.3);
        \node (C1) at (5.7, 2.8) {$\bullet$};
          \node at (5.6,2.3) {$-2$};
          \node (W1) at (7, 2.8) {$\bullet$};
          \draw [dotted] (C1) -- (W1);
          \node at (6.9,2.3) {$-2$};
         \node (C2) at (7.8, 2.8) {$\bullet$};
         \draw [-] (W1) -- (C2);
         \node at (7.7,2.3) {$-2$};
         
    \node (F1) at (5.7, 4.3) {$\bullet$};
          \node at (5.6,4.8) {$-2$};
          \node (G1) at (7, 4.3) {$\bullet$};
          \draw [dotted] (F1) -- (G1);
          \node at (6.9,4.8) {$-2$};
         \node (F2) at (7.8, 4.3) {$\bullet$};
         \draw [-] (G1) -- (F2);
         \node at (7.7,4.8) {$-2$};

 \draw[loosely dotted] (6, 3) .. controls (5.9,3.5) ..  (6, 4.1);
       
    \node (B3) at (9, 0) {$\bullet$};
    \draw [-] (B2) -- (B3);
    \draw [-] (K2) -- (B3);
    \draw [-] (C2) -- (B3);
    \draw [-] (F2) -- (B3);
    \node at (10.5,-3) {$v_c=-e_1^1-e_1^2-\dots-e_1^r-f_1+e_2^1-e_3^1$};
    \draw [->,zigzag](8.6,-2.5) --  (9,-.4);
     \node (B4) at (10, 0) {$\bullet$};
     \draw [color=red] (10,0) ellipse (.15 and -.33);
     \draw [-] (B3) -- (B4);
     \node at (11.5,0) {$v_r=e_1^r+\dots+e^r_{|n_r|}$};

      \node (K4) at (10, -1) {$\bullet$};
     \draw [color=red] (10,-1) ellipse (.15 and -.33);
     \draw [-] (B3) -- (K4);
     \node at (11.1,-1) {$w'=f_1+f_2$};

    \draw[loosely dotted] (10, 0.4) .. controls (10.1,1.2) ..  (9.9, 1.9);
     
         \node (D2) at (9.8, 2.3) {$\bullet$};
         \draw [color=red] (9.8,2.3) ellipse (.15 and -.33);
           \node at (11.2,2.4) {$v_{2}=e_1^{2}+\dots+e^{2}_{|n_{2}|}$};
         \draw (D2) -- (B3);

         \node (L2) at (9.2, 3.2) {$\bullet$};
         \draw [color=red] (9.2,3.2) ellipse (.15 and -.33);
           \node at (10.6,3.5) {$v_{1}=e_1^{1}+\dots+e^{1}_{|n_{1}|}$};
         \draw (L2) -- (B3);        
   \end{tikzpicture}
\caption{Partial embedding of $\Gamma_1$ into the standard negative diagonal lattice with basis $\{e^1_1,\dots,e^r_{|n_r|},f_1,f_2\}$. The embedding of the remaining $-2$-chains uses only basis vectors of the form $e^1_*$.}
\label{f:partial1}
\end{figure}
%

We are left then with $q := t-(r-1)$ further $U^j$'s, the embeddings of which contain only vectors in $\iota(v_1)$, and such that the embedding of the initial vertex of each contains one of $\{e^1_1, e^1_2, e^1_3\}$. Appealing again to Observation \ref{ob_uniquely_used_pair}, we see that $1\leq q$, and by orthogonality $q\leq 3$. Moreover, as $d+t-r=2$, we see that $q=3-d$; as $d\geq1$ by assumption, we discard the possibility $q=3$.

Suppose then that $q=1$.  Note that in this case the associated graph of $\iota(v_1)$ has vertex set $U^j \cup \{v_c,v_1\}$ for some $j$, and by Observation \ref{ob_uniquely_used_pair}, $p_j=|n_1|=:n$. It follows that this graph $\Gamma$ has $n+1$ vertices, and gives a lattice embedding $(\Z^{n+1},Q_\Gamma)$ into $(\Z^{n},-\mathrm{Id})$. On the other hand, the absolute value of the weight of each vertex in $\Gamma$ is at least its valency, from which we conclude that the matrix $Q_\Gamma$ is diagonally dominant and thus has non-zero determinant, giving a contradiction to the `algebraic version' of Remark 3.1. 

We arrive at the same contradiction if we analyze the final case with $q=2$ and $d=1$: the associated graph of $\iota(v_1)$ has vertex set $U^{j_1}\cup U^{j_2} \cup \{v_c,v_1\}$ for some $j_1,j_2$. Call this three legged star-shaped graph $\Gamma$. Its vertex set is obtained from that of  $\Gamma_1$ by deleting the vertices $v_2,\dots,v_r,w'$ and all $u_i^j$ with $j\not\in\{j_1,j_2\}$. The rank of $\Gamma$ is then $p_{j_1}+p_{j_2}$. In the current set-up, Equation~\eqref{e:rank} combined with the fact that for $i=2,\dots,r$ we have $|n_i|=p_j$ for a unique $j$ depending on $i$ reads 
\begin{equation}
    \sum_{i=1}^r|n_i|+2=\sum_{i=1}^t p_i +d \Leftrightarrow |n_1|+2=p_{j_1}+p_{j_2}+1.
\end{equation}
It follows that the rank of $\Gamma$ is $|n_1|+1$. However, $Q_\Gamma$ is diagonally dominant and we have an embedding of $\Gamma$ into a standard negative lattice of rank $|n_1|$ which again contradicts the `algebraic version' of Remark~\ref{r:useful}.

To finish the proof we still need to discard the possibility of an embedding of the lattice $(\Gamma_2,Q_{\Gamma_2})$. This is much easier than the preceding case. In $\Gamma_2$ the central vertex has weight $d-r$, which, since $d\geq 1$, yields $|v_c\cdot v_c|<r$. It is immediate to check, and follows from the arguments above, that we need at least $r$ basis vectors to attach $r$ $-2$-chains of length greater than two. This necessitates a central vertex with $|v_c\cdot v_c|\geq r$. This contradiction finishes the proof.

\end{proof}

Now that we have established Proposition~\ref{p:2A} we proceed to analyze the case of Type 2B pretzel knots with unitary parameters. Recall from Gabai's classification, Theorem~\ref{t:Gabai}, that a Type 2B fibered pretzel knot has an equal number of positive and negative odd parameters. Continuing with our convention of using $d$ for the number of $1$s, $t$ for the number of odd parameters greater than one and $r$ for the number of negative odd parameters, we can write Gabai's condition as $d+t=r$. Furthermore, for a Type 2B pretzel knot to be fibered we need either $L'=\pm P(2,-2,\dots,2,-2,n)$, $n\in\Z$, which is impossible since in our case $t\neq r$; or $L'=\pm P(2,-2,\dots,2,-2,2,-4)$, which is also impossible since this link has an even number of parameters defining it, so $d$ is odd and therefore $t\neq r$, but, at the same time, $L'$ has the same number of $-2$ and $2$ entries, which forces $t=r$. We collect this observation in the next proposition.
\begin{prop}\label{p:2B}
There are no Type 2B fibered pretzel knots with unitary parameters.
\end{prop}

Since the case of Type 2C pretzel knots reduces to Type 3, we are only left with the analysis of the latter family in the presence of unitary parameters. When looking at fiberedness combined with sliceness within Type 3 pretzel knots with unitary parameters, we do find some knots in the intersection. The arguments used to establish this fact are not completely dissimilar from those used to settle the case of Type 2 knots; however, the fine detail producing the examples is unique to the Type~3 knots. Since the arguments used involve Seifert spaces, for which the order and overall sign of the parameters do not matter, we will consider a `model' Type~3 pretzel knot to be of the form $K=P([1^d],p_1,\dots,p_t,n_1,\dots,n_r)$, where $d\geq 1$, $p_i\geq 2$, $n_i\leq -2$ and there is a unique even parameter. The associated Seifert space is $Y_K([1^d],p_1,\dots,p_t,n_1,\dots,n_r)$. In the next proposition we deal with Type 3A.
\begin{prop}\label{p:3A}
Let $K$ be a Type 3A fibered pretzel knot with unitary parameters. If $K$ is slice, then up to mutation:
$$
K=P(1,t+1,3,-4-t,q_1,\dots,q_r,-q_1,\dots,-q_r),
$$
with $q_i\geq 3$ odd, $r,t\geq 0$.
\end{prop}

\begin{proof}
    As in the Type 2A case, we have two possible negative definite graphs, $\Gamma_1$ and $\Gamma_2$, associated with these knots. These are depicted in Figure~\ref{f:3Aones} together with the Wu set. The essential difference between Type 3A and Type 2A graphs is that in the former there is one even parameter among the $p_i$ and $n_j$, while in the latter the unique even parameter had weight $\pm2$ and was considered independently of the other parameters.
\begin{figure}
  \begin{tikzpicture}[xscale=1.3,yscale=-0.6]
    \node (A0_4) at (4.2, 3.8) {$n_r$};
    \node (A1_4) at (4, 4) {$\bullet$};
    \node (A2_2) at (1.7, 3.8) {$p_1$};
    \node (A2_2) at (1.7, 5) {$p_t$};
    \node (A2_3) at (3, 3.4) {$-d$};
    \node (A3_1) at (2, 5) {$\bullet$};
    \node (A3_2) at (2, 4) {$\bullet$};
    \node (A3_3) at (3, 4) {$\bullet$};
    \node (B) at (3.2,5.5) {$\bullet$};
    \node at (3.2,5.9) {$n_1$};
    \node at (4.4, 5.1) {$n_{r-1}$};
    \node (A4) at (4, 5) {$\bullet$};
     \path (A3_3) edge [-] node [auto] {$\scriptstyle{}$} (A4);
     \path (A3_3) edge [-] node [auto] {$\scriptstyle{}$} (A3_1);
     \path (A3_3) edge [-] node [auto] {$\scriptstyle{}$} (A1_4);
     \path (A3_3) edge [-] node [auto] {$\scriptstyle{}$} (A3_2);
     \path (A3_3) edge [-] node [auto] {$\scriptstyle{}$} (B);
     \draw [dotted] (A3_2) -- (A3_1);
     \draw [dotted] (B) .. controls (3.6,5.5) .. (A4);
     
    \path[draw,thin,->] 
    (4.3,2.6) -- ++(0,-1.2cm)  -- (5.4,1.4);
    \node at (5.8,1.4) {$\Gamma_1:$};
    \node at (4.6,1) {\small{$e(Y_K)<0$}};
    
    \path[draw,thin,->] 
    (4.3,6) -- ++(0,1cm)  -- (5.4,7);
    \node at (5.8,7) {$\Gamma_2:$};
    \node at (4.6,7.5) {\small{$e(Y_K)>0$}};
   
          \node (B1) at (6.8, 0) {$\bullet$};
          \node at (6.8,-0.4) {$-2$};
         \node (B2) at (8.1, 0) {$\bullet$};
         \draw [dotted] (B1) -- (B2);
         \node at (8.1,-0.4) {$-2$};
         \draw[decorate,decoration={brace,amplitude=5pt,raise=.5cm},yshift=0pt] (6.7, 0) -- (8.2, 0) node [midway,yshift=.9cm]{$p_{1}\!-\!1$};
    \node (B3) at (9, 0) {$\bullet$};
    \draw [-] (B2) -- (B3);
    \node at (9,-0.4) {$-d\!-\!t$};
     \node (B4) at (10, 0) {$\bullet$};
     \draw [color=red] (10,0) ellipse (.15 and -.33);
     \draw [-] (B3) -- (B4);
     \node at (10.3,-0.4) {$n_r$};

      \node (C1) at (7, 3) {$\bullet$};
          \node at (6.9,2.6) {$-2$};
         \node (C2) at (8.2, 2) {$\bullet$};
         \draw [dotted] (C1) -- (C2);
         \node at (7.9,1.8) {$-2$};
          \draw[loosely dotted] (6.8, 0.4) .. controls (6.7,1.4) ..  (7, 2.3);
         \draw[decorate,decoration={brace,amplitude=5pt,raise=.2cm,mirror},yshift=0pt] (7, 3) -- (8.2, 2) node [midway,sloped,rotate=60,below,yshift=-.4cm]{$p_{t}\!-\!1$};
         \draw [-] (C2) -- (B3);
        \node (D1) at (9.3, 2.2) {$\bullet$};
        \draw [color=red] (9.3,2.2) ellipse (.15 and -.33);
          \node at (9.7,2.4) {$n_1$};
           \draw  (D1) -- (B3);
         \node (D2) at (10, 1) {$\bullet$};
         \draw [color=red] (10,1) ellipse (.15 and -.33);
           \node at (10.5,1.1) {$n_{r-1}$};
        \draw[loosely dotted] (9.5, 2.1) .. controls (9.7,2.1) ..  (9.9, 1.3);
         \draw (D2) -- (B3);
        \node (E1) at (6.8, 6) {$\bullet$};
        \draw [color=red] (6.8,6) ellipse (.15 and -.33);
        \node at (6.4, 5.7) {$-p_1$};
        \node (E2) at (7.8, 6) {$\bullet$};
        \node at (7.8, 5.6) {$d\!-\!r$};
        \draw [-] (E1) -- (E2);
        \node (E3) at (8.8, 6) {$\bullet$};
        \node at (8.8, 5.6) {$-2$};
        \node (X1) at (9.8,6) {$\bullet$};
        \node at (9.8, 5.6) {$-2$};
        \draw [dotted] (E3) -- (X1);
        \draw[decorate,decoration={brace,amplitude=5pt,raise=.55cm},yshift=0pt] (8.8, 6) -- (9.9, 6) node [midway,above,yshift=.7cm]{\tiny{$|n_{r}|\!-\!1$}};
        \draw [-] (E2) -- (E3);
        \node (E4) at (6.9, 7.5) {$\bullet$};
        \draw [color=red] (6.9,7.5) ellipse (.15 and -.33);
        \draw [-] (E4) -- (E2);
        \draw[loosely dotted] (6.7, 6.4) .. controls (6.7,6.8) ..  (6.8, 7.2);
        \node at (6.5, 7.6) {$-p_t$};
        \node (F1) at (8.8, 7) {$\bullet$};
        \node at (8.8, 7.6) {$-2$};
        \node (F2) at (10, 7.9) {$\bullet$};
        \node at (10, 8.4) {$-2$};
        \draw [dotted] (F1) -- (F2);
         \draw [-] (F1) -- (E2);
         \draw[decorate,decoration={brace,amplitude=5pt,raise=.2cm},yshift=0pt] (8.8, 7) -- (10, 7.9) node [midway,sloped,rotate=-54,above,yshift=.3cm]{\tiny{$|n_{r-1}|\!-\!1$}};

         \node (D1) at (7.6, 8) {$\bullet$};
         \node at (7.9, 7.8) {$-2$};
         \node (D2) at (8.8, 9) {$\bullet$};
        \node at (9.1, 8.9) {$-2$};
        \draw [dotted] (D1) -- (D2);
         \draw [-] (D1) -- (E2);
         \draw[decorate,decoration={brace,amplitude=5pt,raise=.2cm,mirror},yshift=0pt] (7.6, 8) -- (8.8, 9) node [midway,sloped,rotate=-62,above,yshift=-.9cm]{\tiny{$|n_{1}|\!-\!1$}};

         \draw[loosely dotted] (7.8, 7) .. controls (8,7.1) ..  (8.3, 6.5);
   \end{tikzpicture}
\caption{Two possible negative definite graphs associated to a Type 3A pretzel knot $K$, depending on the value of $e(Y_K)$. The Wu class in each case is depicted in red, and independent of which is the even parameter.}
\label{f:3Aones}
\end{figure}
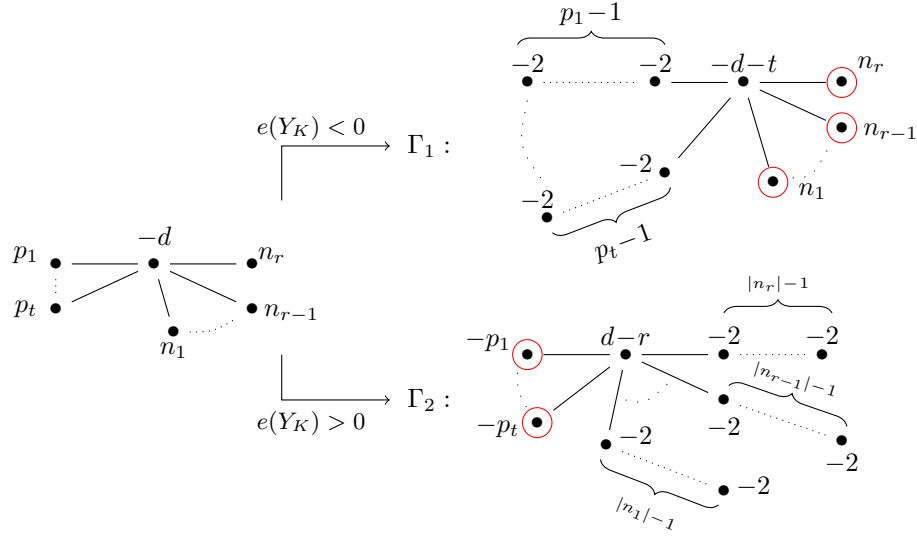
%

%

The case of $\Gamma_2$ is immediate to exclude: there cannot be an embedding of this graph since the linking of the $r$ $-2$-chains, of which there is at most one shortest one of length 1, necessitates a central vertex of weight less or equal than $-r$; however, since $d\geq 1$ we do not have this condition.

The case of $\Gamma_1$ has many similarities with the analogous study in the Type 2A case. The fiberdness condition, $d+t-r=\pm2$, coupled with the linking of the Wu set to the central vertex, $d+t\geq r$, forces $d+t=r+2$. We obtain a partial embedding of $\Gamma_1$ as depicted in Figure~\ref{f:partial3A}. Note the special role played by the basis vectors $e^{i_0}_2$ and $e^{i_0}_3$.
\begin{figure}
   \begin{tikzpicture}[xscale=1.3,yscale=-0.6]
         \node at (6, -.2) {$U^1:$};
        \node (B1) at (6.8, 0) {$\bullet$};
          \node at (6.8,-0.4) {$-2$};
         \node (B2) at (8.1, 0) {$\bullet$};
         \draw [dotted] (B1) -- (B2);
         \node at (8.1,-0.4) {$-2$};
         
    \node (B3) at (9, 0) {$\bullet$};
    \draw [-] (B2) -- (B3);
    \node at (10.5,-2) {$v_c=-e_1^1-e_1^2-\dots-e_1^r+e^{i_0}_2-e_3^{i_0}$};
    \draw [->,zigzag](8.8,-1.5) --  (9,-.6);
     \node (B4) at (10, 0) {$\bullet$};
     \draw [color=red] (10,0) ellipse (.15 and -.33);
     \draw [-] (B3) -- (B4);
     \node at (11.5,0) {$v_r=e_1^r+\dots+e^r_{|n_r|}$};

     \node at (6, 2.8) {$U^t:$};
      \node (C1) at (7, 3) {$\bullet$};
          \node at (6.9,2.6) {$-2$};
         \node (C2) at (8.2, 2) {$\bullet$};
         \draw [dotted] (C1) -- (C2);
         \node at (7.9,1.8) {$-2$};
          \draw[loosely dotted] (6.8, 0.4) .. controls (6.7,1.4) ..  (7, 2.3);
         \draw [-] (C2) -- (B3);
        \node (D1) at (9.3, 2.2) {$\bullet$};
        \draw [color=red] (9.3,2.2) ellipse (.15 and -.33);
          \node at (10.7,2.7) {$v_1=e_1^1+\dots+e^1_{|n_1|}$};
           \draw  (D1) -- (B3);
         \node (D2) at (10, 1) {$\bullet$};
         \draw [color=red] (10,1) ellipse (.15 and -.33);
           \node at (11.9,1.1) {$v_{r-1}=e_1^{r-1}+\dots+e^{r-1}_{|n_{r-1}|}$};
        \draw[loosely dotted] (9.5, 2.1) .. controls (9.7,2.1) ..  (9.9, 1.3);
         \draw (D2) -- (B3);
   \end{tikzpicture}
    \caption{Partial embedding of graph $\Gamma_1$ in a lattice of rank $\sum_{i=1}^r|n_i|$ in the Type 3A family.}
    \label{f:partial3A}
\end{figure}
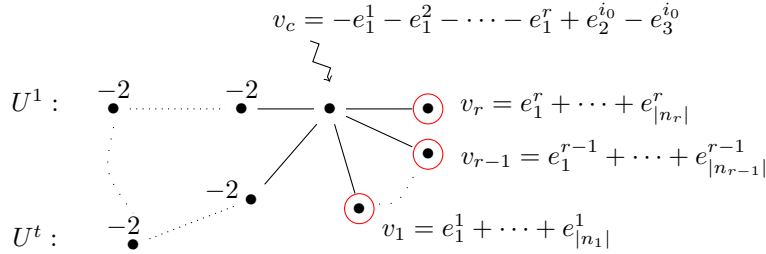

%
Moreover, with the same notation and using the arguments spelled out in the proof of Proposition~\ref{p:2A}, in the current scenario one can show:
        \begin{itemize}
            \item If $\iota(u_1^i)$  is linked to the central vertex via $e^j_k$, then in $\iota(U^i)$ we find only vectors of the form $e^j_*$.
            \item If moreover $n_j$ and $p_i$ are both odd and $j\neq i_0$, then $|n_j|=p_i$ and the $e^j_*$ basis vectors appear only in $v_j$, $v_c$ and $U^i$.
        \end{itemize}
Notice that, if the even parameter is negative, the situation described in the second bullet point applies to all $n_i\neq n_{i_0}$ except for perhaps one exception. Without loss of generality we may assume it applies to all $n_i$ with $i\not\in\{i_0,r\}$; if the even parameter is positive, we may assume it is $p_1$ and, just as before, we have at most one $n_i$, $i\neq i_0$, not falling into the second bullet point. Once again, without loss of generality, we decide this is $n_r$. Summing up, except perhaps for the indices $\{i_0,r\}$, each $n_i$ has, through the embedding, a corresponding positive parameter of the same absolute value. There is a rational homology cobordism (see \cite[Remark~6.5]{pretzel}) between $Y_K([1^d],p_1,\dots,p_t,n_1,\dots,n_r)$ and $Y_{K'}([1^d],p_1,\dots,p_{t'},n_{i_0},n_r)$, where the positive parameters might have been re-indexed and $t'\leq t$. It follows that in order to show that $\Gamma_1$ does not embed, it suffices to prove the same statement for $\Gamma'_1$, which is depicted with partial embedding in Figure~\ref{f:partialbis}.
\begin{figure}[h]
   \begin{tikzpicture}[xscale=1.3,yscale=-0.6]
         \node at (5.7, -.2) {$U^1:$};
        \node (B1) at (6.5, 0) {$\bullet$};
          \node at (6.5,-0.4) {$-2$};
         \node (B2) at (7.8, 0) {$\bullet$};
         \draw [dotted] (B1) -- (B2);
         \node at (7.8,-0.4) {$-2$};
         
    \node (B3) at (9, 0) {$\bullet$};
    \draw [-] (B2) -- (B3);
    \node at (10,-2) {$v_c=-e_1^r-e_1^{i_0}+e^{i_0}_2-e_3^{i_0}$};
    \draw [->,zigzag](8.8,-1.5) --  (9,-.6);
     \node (B4) at (10, 0) {$\bullet$};
     \draw [color=red] (10,0) ellipse (.15 and -.33);
     \draw [-] (B3) -- (B4);
     \node at (11.5,0) {$v_r=e_1^r+\dots+e^r_{|n_r|}$};

     \node at (5.7, 2.8) {$U^{t'}:$};
      \node (C1) at (6.5, 3) {$\bullet$};
          \node at (6.5,2.6) {$-2$};
         \node (C2) at (7.9, 3) {$\bullet$};
         \draw [dotted] (C1) -- (C2);
         \node at (7.8,2.6) {$-2$};
         \draw [-] (C2) -- (B3);

         \node at (5.7, 0.8) {$U^2:$};
      \node (K1) at (6.5, 1) {$\bullet$};
          \node at (6.5,0.6) {$-2$};
         \node (K2) at (7.9, 1) {$\bullet$};
         \draw [dotted] (K1) -- (K2);
         \node at (7.8,0.6) {$-2$};
         \draw [-] (K2) -- (B3);

    \draw[loosely dotted] (7, 1.2) .. controls (6.8,1.9) ..  (7, 2.6);
     
         \node (D2) at (10, 1) {$\bullet$};
         \draw [color=red] (10,1) ellipse (.15 and -.33);
           \node at (11.6,1.1) {$v_{i_0}=e_1^{i_0}+\dots+e^{i_0}_{|n_{i_0}|}$};
         \draw (D2) -- (B3);
   \end{tikzpicture}
    \caption{Partial embedding of graph $\Gamma'_1$ in a lattice of rank $|n_{i_0}|+|n_r|$. While a priori $0\leq t'\leq t$, we establish in the proof that $t'\in\{1,2,3\}$.}
    \label{f:partialbis}
\end{figure}
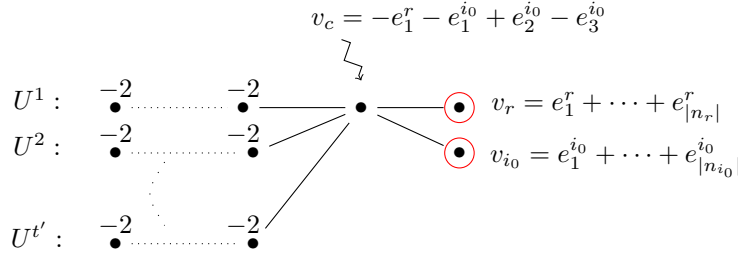
%
The embedding of $\Gamma'_1$ implies that $t'\leq 4$: there is at most one $-2$-chain of shortest length one and therefore each chain has to be linked to the central vertex with a different basis vector. Moreover, since $K$ and $K'$ only differ by a certain number of paired up parameters, for $K'$ we must have $0<d=2+2-t'$ which forces $t'\in\{0,1,2,3\}$. Furthermore, since $K$ and $K'$ are cobordant up to mutation, their signatures coincide, and applying Equation~\eqref{e:sign} to $K'$ and $\Gamma'_1$ we obtain 
    \begin{equation}\label{e:Gammaprime}
          |n_{i_0}|+|n_r|=3+\sum_{j=1}^{t'} (p_j - 1).
    \end{equation}
This last expression yields an immediate contradiction in the case $t'=0$, since we have  $|n_{i_0}|+|n_r|>4$. We finish the proof studying separately the three possible values of $t'$.
  \begin{itemize}
      \item $t'=1$. In this case there is only one $-2$-chain whose embedding will consist solely of basis vectors of the form $e^r_*$ or $e^{i_0}_*$. The chain has length $p_1-1$ and its embedding necessitates $p_1$ basis vectors, while Equation~\eqref{e:Gammaprime} implies that $p_1 = |n_{i_0}|+|n_r|-2$. If these vectors were all of the form $e^r_*$ we have
      $$
      |n_r|\geq |n_{i_0}|+|n_r|-2\Rightarrow|n_{i_0}|\leq2,
      $$
      and this cannot happen, since we have at least 3 basis vectors in the embedding of $v_{i_0}$.  Suppose then that they are all of form $e^{i_0}_*$. By orthogonality $u^1_1$ has one basis vector in common with $v_c$, and similarly $u^1_2$ has either both or neither of its basis vectors in $E:=\{e_1^{i_0},e^{i_0}_2,e^{i_0}_3\}$. If neither, then the same is true of each $u^1_k$, $k>1$, so $p_1 \leq |n_{i_0}| - 2$. If both, then again the same is true of each $u_k$, but then as there are only three basis vectors available the chain has length at most two, and uses only two vectors from $E$. In either case, $|n_{i_0}| > p_1$, so 
      $$
      |n_{i_0}|-1\geq |n_{i_0}|+|n_r|-2\Rightarrow|n_r|\leq 1,
      $$
      which again cannot happen since $|n_r|\geq 2$. 
      \item $t'=2$. We now have two $-2$-chains which we need to embed. 
      Using Observation \ref{ob_uniquely_used_pair}, we need to use at least one $e^r_*$, and thus have one $-2$-chain $U^1$ embedded with $e^r_*$ vectors and the other $U^2$ with $e^{i_0}_*$ vectors. By Equation~\eqref{e:Gammaprime} we have that there is a total of $|n_{i_0}|+|n_r|-3$ weight $-2$ vertices and, since there are two chains, we need precisely $|n_{i_0}|+|n_r|-1$ vectors in their embedding. Since, as remarked in the previous case, $U^2$ cannot use all of $U^{i_0}_*$, it must use all but one of these, while $U^1$ uses all of the $e^r_*$. We may again then invoke the cobordism argument to delete $U^1$ and $v_r$, and remove $e^r_1$ from the embedding of $v_c$. We are thus left with embedding the graph $\Gamma$ depicted in Figure~\ref{f:lens} with the partial embedding shown. 

\begin{figure}[h]
   \begin{tikzpicture}[xscale=1.3,yscale=-0.6]
        \node (B1) at (6.8, 0) {$\bullet$};
          \node at (6.8,-0.5) {$-2$};
         \node (B2) at (8.1, 0) {$\bullet$};
         \draw [dotted] (B1) -- (B2);
         \node at (8.1,-0.5) {$-2$};
          \draw[decorate,decoration={brace,amplitude=5pt,raise=.5cm},yshift=0pt] (6.7, 0) -- (8.2, 0) node [midway,yshift=.9cm]{\small{$|n_{i_0}|\!-\!2$}};
         
    \node (B3) at (9, 0) {$\bullet$};
    \node at (9,-0.5) {$-3$};
    \draw [-] (B2) -- (B3);
    \node at (9.5,1.5) {$v_c=-e_1^{i_0}+e_2^{i_0}-e_3^{i_0}$};
    \draw [->,zigzag](8.7,1.2) --  (9,0.3);
     \node (B4) at (10, 0) {$\bullet$};
     \node at (10,-0.5) {$n_{i_0}$};
     \draw [-] (B3) -- (B4);
     \node at (11.5,0.5) {$v_{i_0}=e_1^{i_0}+\dots+e^{i_0}_{|n_{i_0}|}$};
\end{tikzpicture}
\caption{Partial embedding of graph $\Gamma$ in a lattice of rank $|n_{i_0}|$ with one $-2$-chain of length $|n_{i_0}|-2$.}
\label{f:lens}
\end{figure}
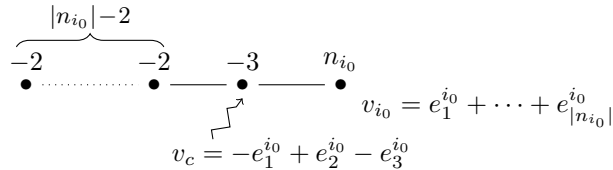
      
%
      The 3-manifold associated to this graph is a lens space. While it is simple to determine directly, we can appeal to Lisca's classification \cite[Lemma~7.3]{lisca} to conclude that $\Gamma$ will embed if and only if $|n_{i_0}|=4$ and the length of the $-2$-chain is 2, that is, for the `pretzel knot' $P(1,1,3,-4)$, which is a 2-bridge ribbon knot. For the sake of completeness, in Figure~\ref{f:band3A1} (with $t=0$) we show the ribbon move, yielding two unlinked unknots, on this pretzel knot (with its standard pretzel projection). 
          %
\begin{figure}
        \centering
        \includegraphics[width=0.40\textwidth]{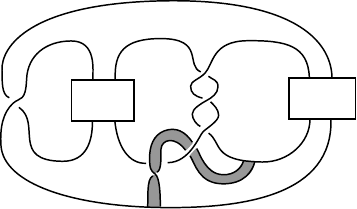}
        \begin{overpic}[width=\linewidth,height=0.001in]{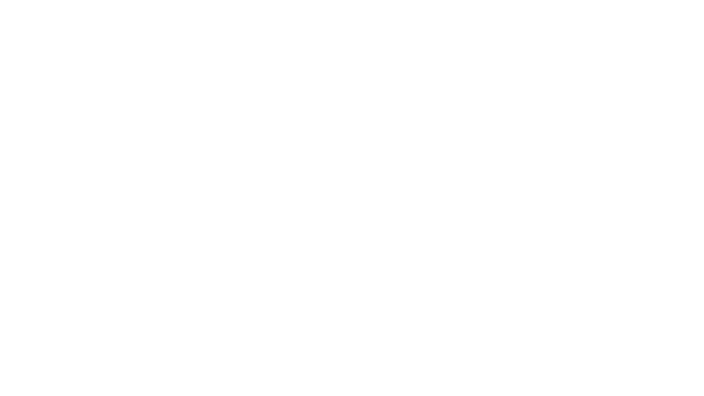}
        
        \put(39,15) {\small $t+1$}
        
        \put(62.8,15) {\small $-4\!-\!t$}
        \end{overpic}
        \caption{
            The family of ribbon pretzel knots $P(1,t+1,3,-4-t)$ and a band which describes the ribbon disks. Notice that, when $t=0$, this is a 2-bridge knot.
        } 
        \label{f:band3A1}
    \end{figure}

      Note then that by adding arbitrary paired parameters discarded in the cobordism argument, we obtain $P(1,1,3,-4,q_1,...,q_r,-q_1,...,-q_r)$,  with $q_i\geq 3$ odd, $r \geq 0$. That ribbonness is preserved when the paired up parameters are adjacent follows immediately from the cobordism determined by the evident ribbon move (See Figure~\ref{f:pretzel}).        \item $t'=3$. In this last subcase we have three $-2$-chains and again by Equation~\eqref{e:Gammaprime} a total of $|n_{i_0}|+|n_r|-3$ vertices of weight $-2$, now needing $|n_{i_0}|+|n_r|$ basis vectors in their embedding. Arguing as in the preceding point, the $e^r_*$ must all appear in one $U^j$ chain and thus our cobordism argument allows us to again reduce to the case we are dealing with a graph $\Gamma$ just like the one in Figure~\ref{f:lens}, but with two $-2$-chains attached to the central vertex instead of just one. Just as before, all basis vectors need to be used in the embedding of these two $-2$-chains. 

       Again following the arguments of the previous case, we see that each of $u_1^1$ and $u_1^2$ has one basis vector in common with those of $v_c$ and that, for each $k$, either each vertex in $\{u^k_l \ | \ l>1\}$ has no basis vectors in common with $v_c$, or each vertex has both its basis vectors in common with $v_c$. As we must use all of the vectors of $v_c$ somewhere, we conclude that one chain satisfies the latter condition, so is length two, with embedding up to relabelling $(e_1^{i_0}-e_4^{i_0},-e_1^{i_0}+e_3^{i_0})$ (notice the symmetric role played by $e_1^{i_0}$ and $e_3^{i_0}$). The other chain, of arbitrary length, uses all the rest of the basis vectors in the embedding of $v_{i_0}$, so up to relabelling is $(-e_2^{i_0}+e_5^{i_0},-e_5^{i_0}+e_6^{i_0},...)$.  The knot associated to this graph is now $P(1,t+1,3,-4-t)$ with $t\geq 1$, ribbonness of which again is demonstrated by the same band move as in the previous case (Figure~\ref{f:band3A1}). Just as in the preceding point, we may add arbitrary pairs of parameters to arrive at the family $P(1,t+1,3,-4-t,q_1,\dots,q_r,-q_1,\dots,-q_r)$, with $q_i\geq 3$ odd, $r \geq 0, t\geq 1$. Again the obvious ribbon cobordism will reduce the paired parameters when these are adjacent.
  \end{itemize}
\end{proof}

Following the strategy laid out for this article, we now proceed to study the case of fibered Type 3B pretzel knots with unitary parameters which, up to an overall change of sign, we can assume to be positive. If such a knot is fibered, the number of positive and negative parameters coincide. The auxiliary link in Gabai's Theorem~\ref{t:Gabai} has a difference of one between the number of positive and negative entries. It follows that a Type 3B fibered pretzel knot with positive unitary entries has a unique parameter of value 1. Moreover, a cyclic permutation of the rest of the parameters, of which we have an odd number, has them alternating in sign, with the last parameter negative. Some examples of these fibered knots are $P(1,5,-3,-4)$ and $P(1,-3,-7,4,-5,5)=P(1,4,-5,5,-3,-7)$. 

Recall we are assuming that our pretzel knots do not have both parameters $\pm1$ and $\mp 2$ since, as explained in Section~\ref{s:background}, these knots admit simpler pretzel knot descriptions. This fact will be key in the proof of the next proposition. Once again, since we will be dealing with Seifert spaces in the arguments, for which the order of the parameters in the pretzel knot do not matter, we will consider for Type~3B fibered pretzel knots $K$ the Seifert spaces $Y_K(1,p_1,\dots,p_t,n_1,\dots,n_{t+1})$ where $p_i\geq 2$, $n_i\leq -3$ and there is a unique even parameter.
\begin{prop}\label{p:3B}
    Type 3B fibered pretzel knots with $\pm 1$ parameters are not slice.
\end{prop}
\begin{proof}
    We start by establishing the two possible graphs, $\Gamma_1$ and $\Gamma_2$ associated with $Y_K$, the Seifert spaces described in the paragraph preceding this proposition. They are depicted in Figure~\ref{f:3Bones} together with the Wu set. Notice that the Wu set and the graphs are independent of which amongst the $p_i$ and $n_j$ is the even parameter. 

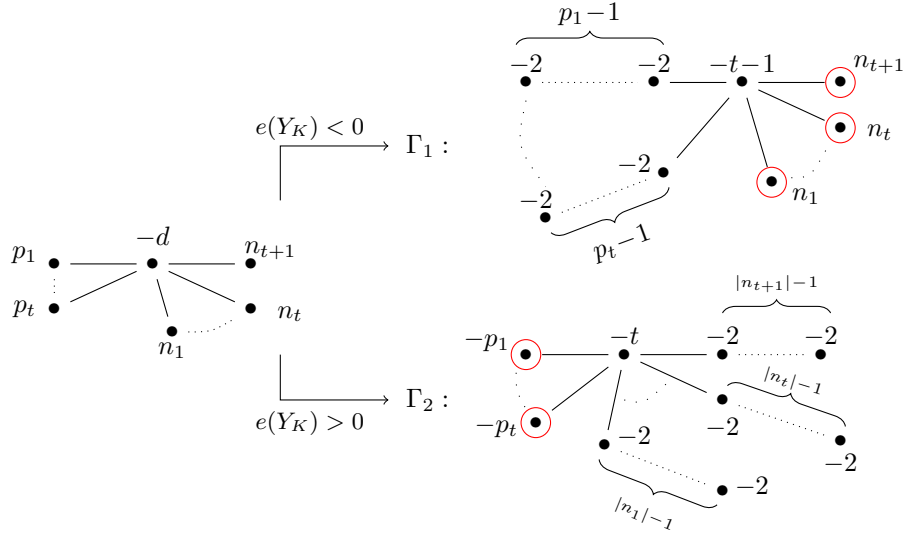
\begin{figure}
  \begin{tikzpicture}[xscale=1.3,yscale=-0.6]
    \node (A0_4) at (4.2, 3.7) {$n_{t+1}$};
    \node (A1_4) at (4, 4) {$\bullet$};
    \node (A2_2) at (1.7, 3.8) {$p_1$};
    \node (A2_2) at (1.7, 5) {$p_t$};
    \node (A2_3) at (3, 3.4) {$-d$};
    \node (A3_1) at (2, 5) {$\bullet$};
    \node (A3_2) at (2, 4) {$\bullet$};
    \node (A3_3) at (3, 4) {$\bullet$};
    \node (B) at (3.2,5.5) {$\bullet$};
    \node at (3.2,5.9) {$n_1$};
    \node at (4.4, 5.1) {$n_{t}$};
    \node (A4) at (4, 5) {$\bullet$};
     \path (A3_3) edge [-] node [auto] {$\scriptstyle{}$} (A4);
     \path (A3_3) edge [-] node [auto] {$\scriptstyle{}$} (A3_1);
     \path (A3_3) edge [-] node [auto] {$\scriptstyle{}$} (A1_4);
     \path (A3_3) edge [-] node [auto] {$\scriptstyle{}$} (A3_2);
     \path (A3_3) edge [-] node [auto] {$\scriptstyle{}$} (B);
     \draw [dotted] (A3_2) -- (A3_1);
     \draw [dotted] (B) .. controls (3.6,5.5) .. (A4);
     
    \path[draw,thin,->] 
    (4.3,2.6) -- ++(0,-1.2cm)  -- (5.4,1.4);
    \node at (5.8,1.4) {$\Gamma_1:$};
    \node at (4.6,1) {\small{$e(Y_K)<0$}};
    
    \path[draw,thin,->] 
    (4.3,6) -- ++(0,1cm)  -- (5.4,7);
    \node at (5.8,7) {$\Gamma_2:$};
    \node at (4.6,7.5) {\small{$e(Y_K)>0$}};
   
          \node (B1) at (6.8, 0) {$\bullet$};
          \node at (6.8,-0.4) {$-2$};
         \node (B2) at (8.1, 0) {$\bullet$};
         \draw [dotted] (B1) -- (B2);
         \node at (8.1,-0.4) {$-2$};
         \draw[decorate,decoration={brace,amplitude=5pt,raise=.5cm},yshift=0pt] (6.7, 0) -- (8.2, 0) node [midway,yshift=.9cm]{$p_{1}\!-\!1$};
    \node (B3) at (9, 0) {$\bullet$};
    \draw [-] (B2) -- (B3);
    \node at (9,-0.4) {$-t\!-\!1$};
     \node (B4) at (10, 0) {$\bullet$};
     \draw [color=red] (10,0) ellipse (.15 and -.33);
     \draw [-] (B3) -- (B4);
     \node at (10.4,-0.4) {$n_{t+1}$};

      \node (C1) at (7, 3) {$\bullet$};
          \node at (6.9,2.6) {$-2$};
         \node (C2) at (8.2, 2) {$\bullet$};
         \draw [dotted] (C1) -- (C2);
         \node at (7.9,1.8) {$-2$};
          \draw[loosely dotted] (6.8, 0.4) .. controls (6.7,1.4) ..  (7, 2.3);
         \draw[decorate,decoration={brace,amplitude=5pt,raise=.2cm,mirror},yshift=0pt] (7, 3) -- (8.2, 2) node [midway,sloped,rotate=60,below,yshift=-.4cm]{$p_{t}\!-\!1$};
         \draw [-] (C2) -- (B3);
        \node (D1) at (9.3, 2.2) {$\bullet$};
        \draw [color=red] (9.3,2.2) ellipse (.15 and -.33);
          \node at (9.65,2.5) {$n_1$};
           \draw  (D1) -- (B3);
         \node (D2) at (10, 1) {$\bullet$};
         \draw [color=red] (10,1) ellipse (.15 and -.33);
           \node at (10.4,1.1) {$n_{t}$};
        \draw[loosely dotted] (9.5, 2.1) .. controls (9.7,2.1) ..  (9.9, 1.3);
         \draw (D2) -- (B3);
        \node (E1) at (6.8, 6) {$\bullet$};
        \draw [color=red] (6.8,6) ellipse (.15 and -.33);
        \node at (6.4, 5.7) {$-p_1$};
        \node (E2) at (7.8, 6) {$\bullet$};
        \node at (7.8, 5.6) {$-t$};
        \draw [-] (E1) -- (E2);
        \node (E3) at (8.8, 6) {$\bullet$};
        \node at (8.8, 5.6) {$-2$};
        \node (X1) at (9.8,6) {$\bullet$};
        \node at (9.8, 5.6) {$-2$};
        \draw [dotted] (E3) -- (X1);
        \draw[decorate,decoration={brace,amplitude=5pt,raise=.55cm},yshift=0pt] (8.8, 6) -- (9.9, 6) node [midway,above,yshift=.7cm]{\tiny{$|n_{t+1}|\!-\!1$}};
        \draw [-] (E2) -- (E3);
        \node (E4) at (6.9, 7.5) {$\bullet$};
        \draw [color=red] (6.9,7.5) ellipse (.15 and -.33);
        \draw [-] (E4) -- (E2);
        \draw[loosely dotted] (6.7, 6.4) .. controls (6.7,6.8) ..  (6.8, 7.2);
        \node at (6.5, 7.6) {$-p_t$};
        \node (F1) at (8.8, 7) {$\bullet$};
        \node at (8.8, 7.6) {$-2$};
        \node (F2) at (10, 7.9) {$\bullet$};
        \node at (10, 8.4) {$-2$};
        \draw [dotted] (F1) -- (F2);
         \draw [-] (F1) -- (E2);
         \draw[decorate,decoration={brace,amplitude=5pt,raise=.2cm},yshift=0pt] (8.8, 7) -- (10, 7.9) node [midway,sloped,rotate=-54,above,yshift=.3cm]{\tiny{$|n_{t}|\!-\!1$}};

         \node (D1) at (7.6, 8) {$\bullet$};
         \node at (7.9, 7.8) {$-2$};
         \node (D2) at (8.8, 9) {$\bullet$};
        \node at (9.1, 8.9) {$-2$};
        \draw [dotted] (D1) -- (D2);
         \draw [-] (D1) -- (E2);
         \draw[decorate,decoration={brace,amplitude=5pt,raise=.2cm,mirror},yshift=0pt] (7.6, 8) -- (8.8, 9) node [midway,sloped,rotate=-62,above,yshift=-.9cm]{\tiny{$|n_{1}|\!-\!1$}};

         \draw[loosely dotted] (7.8, 7) .. controls (8,7.1) ..  (8.3, 6.5);
   \end{tikzpicture}
\caption{Two possible negative definite graphs associated to a Type 3B fibered pretzel knot $K$, depending on the value of $e(Y_K)$. The Wu class in each case is depicted in red.}
\label{f:3Bones}
\end{figure}
%
%
    The graph $\Gamma_2$ does not embed (into a standard negative lattice of the same rank) and therefore the associated pretzel knots are not slice. The lack of embedding comes from the fact that we need a central vertex $v_c$ of weight 
    less or equal than $-t-1$ to link the $t+1$ chains of $-2$s of length greater than one, however, $v_c\cdot v_c=-t$. 

    The embedding of the graph $\Gamma_1$, with the same flow of ideas (and notation) described carefully in the proof of Proposition~\ref{p:2A}, satisfies the following:
    \begin{itemize}
        \item The Wu set embeds as $v_i=e^i_1+\dots+e^i_{|n_i|}$ for $i=1,\dots,t+1$.
        \item The central vertex embeds as $v_c=-e_1-\dots-e_{t+1}$.
        \item Each $-2$-chain is linked to the central vertex through a basis vector of the form $e^i_1$. This linking vector determines that in the embedding of the $-2$-chain only basis vectors of the form $e^i_*$ appear. Moreover, if a $-2$-chain embeds with $e^i_*$ vectors, no other $-2$-chain uses these vectors in its embedding. It follows that for each $j=1,\dots,t$, there is a unique $i\in\{1,\dots,t+1\}$ such that $p_j\leq |n_i|$.
        \item Equation~\eqref{e:sign} in this context reads
        $$
        \sum_{i=1}^{t+1}|n_i|=t+1 + 1 +\sum_{j=1}^t(p_j-1)\Leftrightarrow\sum_{i=1}^{t+1}|n_i|=2+\sum_{j=1}^t p_j.
        $$
    \end{itemize}
From the third bullet point, we see that, for some index $k$, corresponding to the `unmatched' $n$ parameter, we have
$$
\sum_{j=1}^t p_j + |n_k| \leq \sum_{i=1}^{t+1}|n_i|.
$$
Combining this with the last bullet point however gives $|n_k|\leq 2$, contradicting the assumption $n_i\leq -3$ for all $i$.

\end{proof}

\begin{re}
    If one follows carefully the proof of the last proposition, one might be left to wonder what happens if we consider the family of Type 3B fibered pretzel knots with $\pm1$ parameters and include the case of $\mp 2$ being the only even parameter. There are indeed many slice, in fact ribbon, fibered pretzels of this sort, one being for example $K=\pm P(1,-3,3,-2)$. The whole family is the obvious one at this point, with an arbitrary number of paired up parameters beyond the 1 and the $-2$. We are discarding these examples since they all admit a simpler presentation as Type 2B slice fibered pretzel knots. In the example proposed we would get $K=\pm P(-3,3,2)$.
\end{re}

We are left only with the sliceness analysis of Type 3C fibered pretzel knots with unitary parameters. However, this is the empty set: by Theorem~\ref{t:Gabai} we are looking at pretzel knots with an even number of parameters, half of them positive, half of them negative and with auxiliary link $L'=\pm P(2,-2,\dots,2,-2)$. The auxiliary link is determined by the signs of the non-unitary entries. If there were unitary parameters, say positive, we cannot get the correct $L'$ while having the same number of positive and negative parameters in $K$. For reference, this immediate observation is collected in the next statement.
\begin{prop}\label{p:3C}
    There are no Type 3C fibered pretzel knots with unitary parameters.
\end{prop}

Theorem~\ref{thm_main} now follows as a combination of Propositions  \ref{p:odd}, \ref{prop_nonunitary}, \ref{p:2A}, \ref{p:2B}, \ref{p:3A}, \ref{p:3B} and~\ref{p:3C}.

\bibliographystyle{alpha}
\bibliography{refs}

\end{document}